\documentclass[]{amsart}

\usepackage{amsmath}
\usepackage{amsthm}
\usepackage{amsfonts}
\usepackage{amssymb}
\usepackage[shortlabels]{enumitem}
\usepackage{todonotes}
\usepackage{float}
\usepackage{mathtools}
\usepackage{csquotes}
\usepackage{changes}
\usepackage{comment}
\usepackage{tikz-cd}
\tikzcdset{every label/.append style = {font = \normalsize}}

\usepackage[style=numeric,sorting=nyt,backend=bibtex8, maxbibnames=99]{biblatex}
\bibliography{reference.bib}
\defbibheading{bibliography}[\refname]{%
  \section*{#1}%
}

\usepackage{hyperref}
\hypersetup{colorlinks=true, citecolor=blue,urlcolor=black, linkcolor=blue}

\usepackage{graphicx}

\makeatletter
\newcommand{\oast}{\mathbin{\mathpalette\make@circled\ast}}
\newcommand{\make@circled}[2]{%
  \ooalign{$\m@th#1\smallbigcirc{#1}$\cr\hidewidth$\m@th#1#2$\hidewidth\cr}%
}
\newcommand{\smallbigcirc}[1]{%
  \vcenter{\hbox{\scalebox{0.77778}{$\m@th#1\bigcirc$}}}%
}
\makeatother

\newenvironment{why}[1][Proof]{\proof[#1]\mbox{}}{\endproof}
\newtheorem{theorem}{Theorem}[section]
\newtheorem{mainresult}{Theorem}

\newtheorem{maintheorem}[mainresult]{Theorem}
\newtheorem{maincorollary}[mainresult]{Corollary}
\newtheorem*{theorem*}{Theorem}
\newtheorem{proposition}[theorem]{Proposition}
\newtheorem*{Dproblem}{Ditor's Problem}
\newtheorem*{Dproblem2}{Ditor's Problem 2}
\newtheorem{corollary}[theorem]{Corollary}
\newtheorem*{corollary*}{Corollary}
\newtheorem{lemma}[theorem]{Lemma}
\newtheorem{claim}{Claim}[theorem]
\newtheorem*{subclaim*}{Subclaim}

\theoremstyle{definition}

\newtheorem{question}{Question}
\newtheorem{conjecture}[question]{Conjecture}
\newtheorem*{conjecture*}{Conjecture}
\newtheorem*{question*}{Question}
\newtheorem*{questionA}{Ditor's Question A}
\newtheorem{definition}[theorem]{Definition}

\theoremstyle{remark}

\begin{document}

\title{A solution to Ditor's problem}

\author{Lorenzo Notaro}
\address{University of Vienna\\
Institute of Mathematics\\
Kurt G\"{o}del research center\\
Kolingasse 14--16\\
1090 Vienna, Austria}
\curraddr{}
\email{lorenzo.notaro@univie.ac.at}

\begin{abstract}
We settle the long-standing open question whether there exists a $3$-ladder of cardinality $\aleph_2$. Given a positive integer $n$, an $n$-ladder is a lower finite lattice whose elements have at most $n$ lower covers. In 1984, Ditor proved that every $n$-ladder has cardinality at most $\aleph_{n-1}$, and that this cardinal bound is sharp for $n = 1,2$. He then raised the question of whether the bound is attained for $n\ge 3$ as well. An affirmative answer is known to be consistent with $\mathsf{ZFC}$. We prove, relative to the consistency of a Mahlo cardinal, that the question is independent of $\mathsf{ZFC}$. More precisely, we show that the nonexistence of a $3$-ladder of cardinality $\aleph_2$ is equiconsistent with a Mahlo cardinal.
\end{abstract}

\thanks{This research was funded in whole or in part by the Austrian Science Fund (FWF) \href{https://www.fwf.ac.at/en/research-radar/10.55776/ESP1829225}{10.55776/ESP1829225}. For open access purposes, the author has applied a CC BY public copyright license to any author accepted manuscript version arising from this submission.}

\subjclass[2020]{Primary 03E35, Secondary 03E05, 03E55, 06A07}
\keywords{lattice, join-semilattice, lower cover, breadth, ladder, Ditor's problem, forcing, Mahlo cardinal, Lévy collapse}
\maketitle

Given a positive integer $n$, an \emph{$n$-ladder} is a lower finite lattice whose elements have at most $n$ lower covers. The notion of an $n$-ladder was introduced independently by Ditor~\cite{MR0732199} and Dobbertin~\cite{MR862871} under different names\footnote{Ditor called these lattices \emph{$n$-lattices}, whereas Dobbertin called them \emph{$n$-frames}. We follow the terminology of Gr\"{a}tzer, Lakser, and Wehrung \cite{MR1768850}.}.

Ditor proved that every $n$-ladder has breadth at most $n$. He further proved that the breadth of a join-semilattice, together with the cardinalities of its principal ideals, yields an upper bound on the cardinality of the join-semilattice itself (Theorem~\ref{thm:ditor}). A particular instance of this result is that every lower finite lattice of breadth at most $n$, and hence every $n$-ladder, has cardinality at most $\aleph_{n-1}$; see Section~\ref{sec:prel:ditor}. He then asked whether this bound is attained by $n$-ladders~\cite[Question B]{MR0732199}:

\begin{question*}
For every $n > 0$, is there an $n$-ladder of cardinality $\aleph_{n-1}$?
\end{question*}

The case $n=1$ is immediate: $\omega$ with its usual ordering is a $1$-ladder of cardinality $\aleph_0$. Ditor also gave a positive answer for $n = 2$ by constructing a $2$-ladder of cardinality $\aleph_1$. Since then, $2$-ladders have been used primarily in representation problems in universal algebra, particularly for structures of cardinality $\leq \aleph_1$ (e.g., \cite{MR862871,MR1768850, MR2309879, MR1800815}). We refer the reader to \cite{MR2926318} for a discussion of the obstacles to obtaining analogous applications for $n$-ladders with $n \ge 3$; no such applications are currently known.

Ditor left his question open for $n\ge 3$ and singled out the case $n = 3$ \cite[Problem~1]{MR0732199}:

\begin{Dproblem}
Is there a $3$-ladder of cardinality $\aleph_2$?
\end{Dproblem}

This problem was later recorded in Gr\"{a}tzer’s monograph in the context of congruence representation problems \cite[\S 4.9]{MR2768581}; see also \cite[\S 5]{MR2926318}.

Wehrung \cite{MR2609217} gave the first consistent positive answer to Ditor's Problem. He proved that the existence of a $3$-ladder of cardinality $\aleph_2$ follows from either of two independent set-theoretic assumptions: $\mathsf{MA}_{\omega_1}(\aleph_1\text{-precaliber})$, that is, $\mathsf{MA}_{\omega_1}$ restricted to forcings of precaliber $\aleph_1$; and the existence of an $(\omega_1, 1)$-morass. 

It was later shown in \cite{MR4993406} that, for every $n \ge 3$, the existence of an $n$-ladder of cardinality $\aleph_{n-1}$ follows from $\square_{\omega_1} + \square_{\omega_2} +\ldots + \square_{\omega_{n-2}}$, i.e., from Jensen's $\square_\kappa$ holding at the first $n-2$ uncountable cardinals. In particular, Ditor's Problem has a positive answer under $\square_{\omega_1}$. Moreover, since the axiom of constructibility $\mathsf{V=L}$ implies $\square_\kappa$ for every uncountable cardinal $\kappa$, we conclude that Ditor's question has a positive answer for every $n > 0$ under $\mathsf{V=L}$. 

A different consistent positive answer for all $n > 0$ was obtained in \cite{notaro2026maximalladders}, where it is shown that $\text{Add}(\omega, \omega_\omega)$---that is, Cohen's forcing for adding $\aleph_\omega$ many Cohen reals---forces the existence of $n$-ladders of cardinality $\aleph_{n-1}$ for every $n > 0$. The same paper introduces maximal $n$-ladders, a notion motivated by the work of Ditor and Wehrung and closely connected with Ditor's Problem; more on this in Section~\ref{sec:questions}.

In this paper, we settle Ditor's Problem by proving, relative to the consistency of a Mahlo cardinal, that it is independent of $\mathsf{ZFC}$. In fact, we establish the following stronger result.
\begin{maintheorem}\label{thm:main}
Relative to the consistency of a Mahlo cardinal, it is consistent that, for every $n \ge 3$, there are no lower finite lattices of breadth at most $n$ and cardinality $\aleph_{n-1}$.
\end{maintheorem}

The proof reduces to the case $n=3$. We show that, if $\kappa$ is a Mahlo cardinal, then  in $\mathsf{V}^{\mathrm{Coll}(\omega_1,{<}\kappa)}$ there are no lower finite lattices of breadth $3$ and cardinality $\aleph_2$ (Theorem~\ref{thm:maind}); the conclusion for every $n\ge3$ then follows (see Corollary~\ref{cor:induction}).

The large cardinal hypothesis is optimal in consistency strength. Indeed, the existence of a $3$-ladder of cardinality $\aleph_{2}$  follows from $\square_{\omega_1}$. Hence, in any model of $\mathsf{ZFC}$ in which no such ladder exists, $\square_{\omega_1}$ must fail.  The failure of $\square_{\omega_1}$ is well known to imply that $\omega_2$ is a Mahlo cardinal in the constructible universe $\mathsf{L}$ \cite{MR750828}. Combining this observation with Theorem~\ref{thm:main} yields the following equiconsistency result. 

\begin{maincorollary}\label{cor:main}
The following theories are equiconsistent:
\begin{enumerate}[label={\upshape (\arabic*)}]
\item $\mathsf{ZFC}+ ``$There is a Mahlo cardinal".
\item $\mathsf{ZFC}+ ``$There are no lower finite lattices of breadth $3$ and cardinality $\aleph_2$". 
\item $\mathsf{ZFC}+ ``$There are no $3$-ladders of cardinality $\aleph_2$". 
\end{enumerate}
\end{maincorollary}

Section~\ref{sec:prel} collects the required preliminary material on the breadth of join-semilattices, Ditor's cardinal bound, and Laflamme's game-theoretic characterization of meager filters.  In this section, we also explain how the formulation of Theorem~\ref{thm:main} in terms of breadth, rather than only in terms of $n$-ladders, settles an instance of a more general question posed by Ditor.

In Section~\ref{sec:quotient}, we introduce a local version of the breadth invariant and study its relationship to the breadth of quotients induced by ideals. As a consequence, we show that if a join-semilattice attains Ditor's cardinal bound, then, for every sufficiently small ideal $I$, there is an element $x$ such that $x\vee I$ is a chain (Theorem~\ref{thm:chain}).

In Section~\ref{sec:proj}, we introduce the key notion of projected upper cones and the associated filters. We prove Theorem~\ref{thm:main} in Section~\ref{sec:main} and conclude with some open questions in Section~\ref{sec:questions}.

\section{Preliminaries}\label{sec:prel}
\subsection{Set Theory}\label{sec:prel:set}
The monographs \cites{MR1940513, MR756630} are our references for all classical definitions and notation in set theory. Given a set $X$ and a (possibly finite) cardinal $\kappa$, $[X]^{\kappa}$ and $[X]^{\le \kappa}$ denote, respectively, the family of all subsets of $X$ of cardinality $\kappa$ and the family of all subsets of $X$ of cardinality ${\le}\kappa$. Given two sets $X$ and $Y$, we say that $X$ is \emph{almost included in} $Y$, and write $X \subseteq^* Y$, if $X \setminus Y$ is finite. Similarly, $X =^* Y$ means that $(X \setminus Y) \cup (Y \setminus X)$ is finite. Given two finite sequences $s$ and $t$, we denote by $s^\smallfrown t$ the concatenation of $s$ and $t$.

A cardinal $\kappa$ is  \emph{(strongly) inaccessible} if $\kappa$ is an uncountable regular cardinal such that $2^{\lambda} < \kappa$ for all $\lambda < \kappa$. An inaccessible cardinal $\kappa$ is \emph{Mahlo} if the set
\[
\{\lambda < \kappa : \lambda \text{ is an inaccessible cardinal}\}
\] 
is stationary in $\kappa$.

Now let us recall some basic properties of Lévy's collapse forcing. Given two cardinals $\delta < \kappa$, we denote by $\mathrm{Coll}(\delta, {<}\kappa)$ the forcing notion whose conditions are functions $p$ with $ \mathrm{dom}(p) \subseteq \kappa \times \delta$ such that
\begin{enumerate}
\item $|\mathrm{dom}(p)| < \delta$,
\item $p(\alpha, \xi) < \alpha$ for each $(\alpha, \xi) \in \mathrm{dom}(p)$,
\end{enumerate}
ordered by reverse inclusion: $p \le q$ if and only if $p \supseteq q$. Here are the basic properties of this forcing---see, e.g., \cite[\S 14]{MR2768691} and \cite[Theorem 15.22]{MR1940513}:
\begin{lemma}\label{lemma:collapse}
Given two regular cardinals $\delta < \kappa$, with $\kappa$ inaccessible, the following hold:
\begin{enumerate}[label={\upshape (\arabic*)}]
\item $\mathrm{Coll}(\delta, {<}\kappa)$ has the $\kappa$-cc.
\item For every decreasing sequence $\langle p_\alpha : \alpha < \eta\rangle$ in $\mathrm{Coll}(\delta, {<}\kappa)$ with $\eta < \delta$, the union $\bigcup_{\alpha < \eta} p_\alpha$ is still an element of $\mathrm{Coll}(\delta, {<}\kappa)$. In particular, $\mathrm{Coll}(\delta, {<}\kappa)$ is $\delta$-closed.
\item For every $\delta$-closed separative forcing $\mathbb{P}$ with $|\mathbb{P}| < \kappa$, $\mathrm{Coll}(\delta, {<}\kappa)$ is equivalent to $\mathbb{P} \ast \dot{\mathrm{Coll}}(\delta, {<}\kappa)$, that is, their Boolean completions are isomorphic.
\item $\mathrm{Coll}(\delta, {<}\kappa)$ forces $\delta^{<\delta} = \delta$.
\end{enumerate}
\end{lemma} 
It follows immediately from the definition that the forcing $\mathrm{Coll}(\delta, {<}\kappa)$ collapses all cardinals in the open interval $(\delta, \kappa)$ to $\delta$. Moreover, if  $\delta$ is regular and $\kappa$ is inaccessible, it follows from (1) above that all the cardinals $\ge\kappa$ are preserved by $\mathrm{Coll}(\delta, {<}\kappa)$, and from (2) that all the cardinals $\le \delta$ are also preserved. In particular, if $\kappa$ is an inaccessible cardinal and $G$ is a $\mathsf{V}$-generic filter for $\mathrm{Coll}(\omega_1, {<}\kappa)$, then $\omega_1^\mathsf{V} = \omega_1^{\mathsf{V}[G]}$ and $\kappa = \omega_2^{\mathsf{V}[G]}$.

\subsection{Join-semilattices} 
The monograph \cite{MR2768581} is our reference for all classical definitions and notation in lattice theory. 

Given a poset $(P, \le)$ and some $x \in P$, we denote by $P \downarrow x$ and $P \uparrow x$ the sets $\{y \in P \mid y \le x\}$ and $\{y \in P \mid y \ge x\}$, respectively.  Sometimes, instead of $P \downarrow x$ and $P \uparrow x$, we simply write ${\downarrow} x$ and ${\uparrow} x$ when no ambiguity arises. The set ${\uparrow} x$ is called the \emph{upper cone of $x$}.  Furthermore, given two elements $x,y \in P$, $y$ is a \emph{lower cover of $x$} if $y < x$ and there is no $z \in P$ with $y < z < x$. A subset $C \subseteq P$ is said to be \emph{cofinal} in $P$ if for every $x \in P$ there is $y \in C$ with $x \le y$. Given two posets $P$ and $Q$, a map $f:P \rightarrow Q$ is \emph{order-preserving} if $f(x) \le f(y)$ for every $x,y \in P$ with $x \le y$.

A \emph{join-semilattice} is a nonempty set equipped with a binary operation, denoted by $\vee$, that is associative, commutative, and idempotent; it induces a partial order via $x \le y \iff x \vee y = y$. Equivalently, a join-semilattice is a partially ordered set in which every pair of elements $x,y$ admits a least upper bound, denoted by $x \vee y$. The dual notion is the \emph{meet-semilattice}.  We treat semilattices as algebraic structures or as posets depending on which representation is better suited to the given context.

A nonempty subset $S$ of a join-semilattice $P$ is called a \emph{join-subsemilattice} of $P$ if it is closed under binary joins. Meet-subsemilattices are dually defined. A downward closed join-subsemilattice is called an \emph{ideal}. Note that $P \downarrow x$ is an ideal of $P$ for every $x \in P$; such ideals are known as \emph{principal ideals}. An ideal that does not coincide with the whole poset is a \emph{proper ideal}. 

Given a nonempty subset $S$ of a join-semilattice $P$ and an element $x \in P$, we denote by $x \vee S$ the set $\{x \vee y: y \in S\}$. Note that if $S$ is a join-subsemilattice, then $x \vee S$ is also a join-subsemilattice.

Given a join-semilattice $P$ and a finite nonempty set $F = \{x_0, x_1, \dots, x_{n}\}$, we denote $x_0 \vee x_1 \vee \ldots \vee x_n$ simply by $\bigvee F$. If we write $\bigvee \emptyset$, we are tacitly assuming that $P$ has a least element $0_P$, and therefore that $\bigvee \emptyset = 0_P$. 

Let us briefly review the notion of quotient join-semilattice. Given a join-semilattice $P$, an equivalence relation $\sim$ on $P$ is a \emph{congruence relation} if for all $x_0,x_1,y_0,y_1$ in $P$,
\[
x_0 \sim y_0 \text{ and } x_1 \sim y_1 \Rightarrow x_0 \vee x_1 \sim y_0 \vee y_1.
\]

Given a congruence relation $\sim$ on $P$, we can define the join operator $\vee$ on the quotient $P/{\sim}$ as follows: for every $x,y \in P$,
\[
[x]_\sim \vee [y]_\sim \coloneqq [x\vee y]_\sim.
\]
It is easy to check that this operation is well defined and makes $P/{\sim}$ into a join-semilattice. The resulting join-semilattice $P / {\sim}$ is called the \emph{quotient join-semilattice of $P$ modulo~$\sim$}. We denote the quotient map by $q_\sim : P \rightarrow P/{\sim}$.

Any ideal $I$ of $P$ induces the following natural congruence relation $\sim_I$ on $P$: for $x,y \in P$, $x \sim_I y$ if there exists $z \in I$ such that $x \vee z = y \vee z$. In this case, we simply write $P/I$ and $q_I$ instead of $P/{\sim_I}$ and $q_{\sim_I}$. Not every congruence relation on a join-semilattice is induced by an ideal.

Let us also recall the definition of breadth, a classical numeric invariant of lattice theory. 
\begin{definition}
Let $P$ be a join-semilattice and $n\in\omega$. We say that $P$ has \emph{breadth at most $n$} if, for every nonempty finite subset $X$ of  $P$, there exists $Y\subseteq X$ with at most $n$ elements such that $\bigvee X = \bigvee Y$. The \emph{breadth} of $P$, denoted by $\mathrm{breadth}(P)$, is the least $n \in\omega$ such that $P$ has breadth at most $n$, if such an $n$ exists.
\end{definition}
In fact, there is a more general notion of breadth which is self-dual and purely poset-theoretical \cite[\S 4]{MR0732199}. We say that a join-semilattice has \emph{finite breadth} if it has breadth at most $n$ for some $n\in\omega$. Furthermore, note that a join-semilattice $P$ has breadth $0$ if and only if  $P = \{0_P\}$, and has breadth at most $1$ if and only if it is a linear order. The next lemma is immediate from the definition.

\begin{lemma}
Given a join-semilattice $P$ and $n \in\omega$, the following are equivalent:
\begin{enumerate}[label={\upshape (\arabic*)}]
\item $P$ has breadth at most $n$.
\item For every $X \in [P]^{n+1}$, there exists $Y \in [X]^n$ such that $\bigvee X = \bigvee Y$.
\end{enumerate}
\end{lemma}

Finally, let us introduce a nonstandard notation from \cite{notaro2026maximalladders}. Given a lower finite lattice $P$, an ideal $I\subseteq P$, and an element $x \in P$, let
\[
\pi_I(x) \coloneqq \bigvee \{y \in I : y \le x\}.
\]
We say that $\pi_I(x)$ is the \emph{projection of $x$ onto $I$}. Since every lower finite lattice has a least element and ${\downarrow} x$ is finite, the set $\{y \in I : y \le x\}$ is finite and nonempty; hence its join is well-defined. Equivalently, $\pi_I(x)$ is the greatest element of $I$ which is less than or equal to $x$, i.e.,
\[
\pi_I(x) = \max \{y \in I : y \le x\} = \max (I \cap ({\downarrow} x)).
\]

Fix a lower finite lattice $P$ and an ideal $I \subseteq P$. The following properties of the map $\pi_I$ follow directly from its definition:
\begin{itemize}[]
\item $\pi_I \upharpoonright I$ is the identity;
\item $\pi_I$ is idempotent, that is, $\pi_I \circ \pi_I = \pi_I$; 
\item $\pi_I(x) \le x$ for every $x \in P$;
\item $\pi_I$ is order-preserving.
\end{itemize}
A self-map of a poset satisfying the last three properties is an \emph{interior operator}, equivalently, a closure operator on the order dual. It is also sometimes called a \emph{kernel operator} or a \emph{projection} \cites{MR1902334, MR1975381, MR437330}.

The following two lemmas from \cite{notaro2026maximalladders} establish two less immediate properties of the maps $\pi_I$. 

\begin{lemma}\label{lemma:meetproj}
Given a lower finite lattice $P$, elements $x,y \in P$, and an ideal $I \subseteq P$, $\pi_I(x \wedge y) = \pi_I(x) \wedge \pi_I(y)$.
\end{lemma}
\begin{proof}
Since $\pi_I(x) \le x$ and $\pi_I(y) \le y$, we have $\pi_I(x) \wedge \pi_I(y) \le x \wedge y$. As $I$ is downward closed, $\pi_I(x) \wedge \pi_I(y) \in I$. We conclude  $\pi_I(x) \wedge \pi_I(y) \le \pi_I(x \wedge y)$.

Furthermore, since $\pi_I(x \wedge y) \le x \wedge y \le x,y$ and $\pi_I(x \wedge y) \in I$,  we conclude $\pi_I(x \wedge y) \le \pi_I(x) \wedge \pi_I(y)$. Overall, we have $\pi_I(x \wedge y) = \pi_I(x) \wedge \pi_I(y)$.
\end{proof}

In other words, Lemma~\ref{lemma:meetproj} tells us that $\pi_I : P\rightarrow I$ is a meet-homomorphism. Moreover, if $x,y \in P$ are such that $x \wedge y \in I$, we conclude from Lemma~\ref{lemma:meetproj} that $x \wedge y = \pi_I(x) \wedge \pi_I(y)$.

\begin{lemma}\label{lemma:concproj}
If $P$ is a lower finite lattice and $I,J \subseteq P$ are ideals with $I \subseteq J$, then $\pi_I \circ \pi_J = \pi_I$.
\end{lemma}
\begin{proof}
Fix an $x \in P$. Since $\pi_I \circ \pi_J(x) \in I$ and $\pi_I \circ \pi_J(x) \le x$, we have $\pi_I \circ \pi_J(x) \le \pi_I(x)$. Furthermore,  $\pi_I(x) \le \pi_J(x)$ since $\pi_I(x) \le x$ and  $\pi_I(x) \in I \subseteq J$. Therefore, $\pi_I(x) = \pi_I \circ \pi_I(x) \le \pi_I \circ \pi_J (x)$. Overall, $\pi_I \circ \pi_J (x) = \pi_I (x)$.
\end{proof}

\subsection{Ditor's Theorem}\label{sec:prel:ditor} 
Ditor proved that the breadth, together with the cardinalities of the principal ideals, provides a neat upper bound on the cardinality of the join-semilattice.

\begin{theorem}[{Ditor, \cite[Theorem 5.2]{MR0732199}}]\label{thm:ditor}
Let $n > 0$ and let $\kappa$ be an infinite cardinal. If $P$ is a join-semilattice of breadth at most $n$ whose principal ideals have cardinality $< \kappa$, then
\begin{enumerate}[label={\upshape (\alph*)}]
\item
\( |P| \le \kappa^{+(n-1)} \), and
\item
\( |I| < \kappa^{+(n-1)} \) for every proper ideal $I$ of $P$.
\end{enumerate}
\end{theorem}

As noted by Wehrung, the first inequality of Ditor's Theorem is in fact a fairly direct corollary of Kuratowski's Free Set Theorem \cite{MR48518} (see also \cite[Theorem 46.1]{MR795592}). 

Since every $n$-ladder is, in particular, a lower finite join-semilattice of breadth at most $n$ \cite[Proposition 4.1]{MR0732199}, it follows from Ditor's Theorem that every $n$-ladder has cardinality at most $\aleph_{n-1}$. 

Another consequence of Ditor's Theorem that we shall use repeatedly is that, for every infinite cardinal $\kappa$ and every $n > 0$, every join-semilattice of cardinality $\kappa^{+(n-1)}$ whose principal ideals have cardinality $< \kappa$ has breadth at least $n$. In particular, for lower finite join-semilattices of cardinality $\aleph_{n-1}$, having breadth at most $n$ is equivalent to having breadth exactly $n$. Thus, Theorem~\ref{thm:main} could equivalently have been stated with ``breadth $n$'' in place of ``breadth at most $n$''.

Finally, let us remark that Ditor also asked whether these more general cardinal upper bounds are sharp \cite[Question~A]{MR0732199}:

\begin{questionA}
For every $n > 0$ and every infinite cardinal $\kappa$, is there a join-semilattice of breadth $n$ and cardinality $\kappa^{+(n-1)}$ whose principal ideals have cardinality less than $\kappa$?
\end{questionA}

Ditor's Question~B on ladders, stated in the introduction, is a more demanding version of Question~A when $\kappa=\aleph_0$. Indeed, every $n$-ladder is, in particular, a lower finite join-semilattice of breadth at most $n$. In this sense, $n$-ladders may be regarded as particularly tame examples of lower finite join-semilattices of breadth at most $n$.

Together with the consistent positive results recalled in the introduction, Theorem~\ref{thm:main} implies that the instance of Ditor's Question~A corresponding to $\kappa=\aleph_0$ and $n\geq3$ is independent of $\mathsf{ZFC}$, relative to the consistency of a Mahlo cardinal. The fact that Theorem~\ref{thm:main} concerns lattices rather than join-semilattices is inessential: adjoining a least element to an infinite lower finite join-semilattice yields a lower finite lattice without changing its cardinality or breadth. See Section~\ref{sec:questions} for further discussion.

\subsection{Filters}\label{sec:prel:filters}

Recall that a \emph{filter} on an infinite set $X$ is a collection $\mathcal{F} \subseteq \mathcal{P}(X)$ that satisfies the following conditions:
\begin{enumerate}
\item $\emptyset \not\in \mathcal{F}$ and $X \in \mathcal{F}$,
\item if $U \in \mathcal{F}$ and $U \subseteq^* V \subseteq X$, then $V \in \mathcal{F}$,
\item if $U,V \in \mathcal{F}$, then $U \cap V \in \mathcal{F}$.
\end{enumerate}
Given a filter $\mathcal{F}$, we denote by $\mathcal{F}^+$ the family of all subsets $H \subseteq X$ such that $X \setminus H \not\in \mathcal{F}$. Note that $H \in \mathcal{F}^+$ if and only if $H \cap U$ is infinite for every $U \in \mathcal{F}$. 

The filter of cofinite subsets of $X$ is the \emph{Fréchet filter}; when $X = \omega$, we denote it by $\mathfrak{Fr}$. By definition, every filter contains the Fréchet filter\footnote{The filters considered here are precisely the \emph{free} filters in the usual terminology. Following a convention common in the set-theoretic literature on filters on $\omega$, we omit the adjective; see, e.g., \cite{MR1367134}.}.

Let us recall the Rudin-Blass reducibility preorder on filters---see, e.g., \cites{MR1627310, MR2855877}.
\begin{definition}
Let $\mathcal{F}$ and $\mathcal{G}$ be filters on $X$ and $Y$, respectively. We say that $\mathcal{F}$ is \emph{Rudin-Blass reducible to} $\mathcal{G}$ and write $\mathcal{F} \le_{RB} \mathcal{G}$ if there exists a finite-to-one map $f: Y \rightarrow X$ such that, for each $U \subseteq X$, $U\in \mathcal{F}$ if and only if $f^{-1}(U) \in \mathcal{G}$.
\end{definition}

For notational clarity, the remainder of this subsection is stated for filters on $\omega$, but the following two lemmas hold for filters on arbitrary countably infinite sets. If we endow $\mathcal{P}(\omega)$ with its standard topology induced by the Cantor space ${}^\omega 2$, we can ask topological questions about filters, seen as subsets of the topological space $\mathcal{P}(\omega)$. The following well-known result of Talagrand~\cite{MR579439} gives a very useful combinatorial characterization of meager filters. 

\begin{lemma}[Talagrand]\label{lemma:characterizationmeager}
The following are equivalent for a filter $\mathcal{F}$:
\begin{enumerate}[label={\upshape (\arabic*)}]
\item $\mathcal{F}$ is meager.
\item $\mathfrak{Fr} \le_{RB} \mathcal{F}$.
\item There exists a sequence $\langle A_n : n\in\omega\rangle$ of finite, pairwise disjoint subsets of $\omega$ such that, for every $U \in \mathcal{F}$, $U \cap A_n \neq \emptyset$ for all but finitely many $n$.
\end{enumerate}
\end{lemma}

In what follows, we will use only the equivalent combinatorial characterizations of meagerness given in conditions~(2) and~(3).

By Talagrand's characterization, it is clear that the meager filters form an upward closed subset with respect to the Rudin-Blass reducibility preorder. However, the family of meager filters is also downward closed:

\begin{lemma}[Folklore]\label{lemma:preservmeager}
Given two filters $\mathcal{F}, \mathcal{G}$ with $\mathcal{F} \le_{RB} \mathcal{G}$, $\mathcal{F}$ is meager if and only if $\mathcal{G}$ is meager.
\end{lemma}
\begin{proof}
The ``only if" is a direct consequence of Talagrand's characterization. It remains to prove the ``if" direction. Suppose that $\mathfrak{Fr} \le_{RB}\mathcal{G}$. We must show $\mathfrak{Fr} \le_{RB} \mathcal{F}$. Equivalently, we need to show that there exists a sequence $\langle A_n : n \in\omega\rangle$ of finite, pairwise disjoint subsets of $\omega$ such that every element of $\mathcal{F}$ intersects all but finitely many of the sets $A_n$. 

Fix two finite-to-one maps $f, g : \omega \rightarrow \omega$ witnessing $\mathfrak{Fr} \le_{RB} \mathcal{G}$ and $\mathcal{F} \le_{RB} \mathcal{G}$, respectively. We inductively define an increasing sequence $\langle k_n : n\in\omega\rangle$ of natural numbers as follows. First let $k_0 = 0$. Now suppose that $k_m$ is defined for every $m \le n$. Since both $f$ and $g$ are finite-to-one, the set $g^{-1}(g[\bigcup_{m \le n} f^{-1}(\{k_m\})])$ is finite. Thus, we can pick $k_{n+1} > k_n$ such that $f^{-1}(\{k_{n+1}\})$ is disjoint from the latter set. This completes the recursive definition.

For each $n$, let $A_n \coloneqq g[f^{-1}(\{k_n\})]$. Since $f$ is finite-to-one, the sets $A_n$ are finite. Moreover, it follows directly from our construction that they are pairwise disjoint. 

Pick some $U \in \mathcal{F}$. By the way we chose $f$ and $g$, $g^{-1}(U) \in \mathcal{G}$, and therefore $f[g^{-1}(U)] \in \mathfrak{Fr}$. Let $m$ be such that $\omega \setminus k_m \subseteq f[g^{-1}(U)]$. Then, for every $n \ge m$, $k_n \in f[g^{-1}(U)]$, so $f^{-1}(\{k_n\}) \cap g^{-1}(U) \neq \emptyset$, and hence $A_n \cap U \neq \emptyset$. In particular, $U$ intersects all but finitely many $A_n$. 
\end{proof}

Finally, let us also recall a game-theoretic characterization of meager filters due to Laflamme \cite{MR1367134}. For a given filter $\mathcal{F}$, consider the following game $\mathcal{G}(\mathcal{F})$. In this game, at the $k$-th round, Player I plays a natural number $n_k$, and then Player II plays another natural number $m_k$
\begin{center}
\begin{tabular}{cccccccc}
I & $n_0$ & & $n_1$ & & $n_2$ & & \dots \\
II & & $m_0$ & & $m_1$ & & $m_2$ & \dots\\
\end{tabular}
\end{center}
\noindent with the rule: $n_k \le m_k \le n_{k+1}$ for every $k \in \omega$. At the end of a play, Player II \emph{wins} if and only if  $\bigcup_{k\in\omega} [n_k, m_k) \in \mathcal{F}^+$.
\begin{theorem}[{Laflamme, \cite[essentially Theorem 2.12]{MR1367134}}]\label{thm:laflamme}
For a given filter $\mathcal{F}$, Player II has a winning strategy in $\mathcal{G}(\mathcal{F})$ if and only if $\mathcal{F}$ is meager.
\end{theorem}

Before sketching a proof of Theorem~\ref{thm:laflamme}, let us briefly recall what a strategy is in this context. A strategy for Player II in Laflamme's game $\mathcal{G}(\mathcal{F})$ is a function $\sigma$ that assigns to each finite nonempty sequence $\langle n_0, \ldots n_k\rangle$ of natural numbers a natural number $\sigma(\langle n_0, \ldots n_k\rangle) \ge n_k$. The strategy $\sigma$ is \emph{winning} if for every infinite sequence $\langle n_k : k\in\omega\rangle$ of natural numbers satisfying $n_{k+1} \ge \sigma (\langle n_0, \ldots n_k\rangle)$ for every $k$, we have 
\[
\bigcup_{k\in\omega} [n_k, \sigma (\langle n_0, \ldots, n_k\rangle)) \in \mathcal{F}^+.
\]

\begin{proof}[Proof sketch of Theorem~\ref{thm:laflamme}]
The ``if" direction is a direct consequence of Talagrand's characterization. Indeed, if $\langle A_n : n\in\omega\rangle$ witnesses Lemma~\ref{lemma:characterizationmeager}(3), then Player II wins by playing at each turn $k$, a natural number $m_k$ big enough so that for some $n > k$, $A_n \subseteq [n_k, m_k)$.

For the ``only if" direction, suppose that $\sigma$ is a winning strategy for Player II and let us define inductively a sequence $\langle M_k : k\in\omega\rangle$ of natural numbers as follows: first let $M_0 = 0$; then, for each $k \ge 0$, let
\[
M_{k+1} = \max \big\{\sigma(\langle n_0, \ldots, n_i\rangle) : i \le k \text{ and } \forall j\le i \ (n_j \le M_k)\big\}.
\]
Then, an argument by contraposition shows that the sequence $\langle [M_k, M_{k+1}) : k\in\omega\rangle$ witnesses Lemma~\ref{lemma:characterizationmeager}(3). Indeed, if there were $U \in \mathcal{F}$ and an increasing sequence $\langle k_i : i \in\omega\rangle$ of natural numbers such that $U \cap [M_{k_i}, M_{k_i+1}) = \emptyset$ for every $i$, then Player I would win against the strategy $\sigma$ by playing the sequence $\langle M_{k_i} : i\in\omega\rangle$.
\end{proof}

Laflamme's game will appear in Lemma~\ref{lemma:game}, where it is compared with an auxiliary game arising from the forcing argument for Theorem~\ref{thm:main}.

\section{Local breadth and quotients}\label{sec:quotient}

In this section we study the relationship between the finite breadth of a join-semilattice and its ideal quotients. Let us start by introducing a local version of breadth.

\begin{definition}
Let $P$ be a join-semilattice and $n\in\omega$. Given $x \in P$, we say that $P$ has \emph{breadth at most $n$ at $x$} if, for every nonempty finite subset $X$ of  $P$ with $\bigvee X = x$, there exists $Y\subseteq X$ with at most $n$ elements such that $\bigvee Y = x$. The \emph{breadth} of $P$ \emph{at $x$}, denoted by $\mathrm{breadth}_P(x)$, is the least $n\in\omega$ such that $P$ has breadth at most $n$ at $x$, if such an $n$ exists.
\end{definition}

For a given $x \in P$, note that $\mathrm{breadth}_P(x) = 0$ if and only if $x = 0_P$, and $\mathrm{breadth}_P(x) = 1$ if and only if $x$ is join-irreducible---i.e., $x \neq 0_P$ and $x = y \vee z$ implies that either $x = y$ or $x = z$. Furthermore, it follows readily that for every join-semilattice $P$ of finite breadth
\[
\mathrm{breadth}(P) = \max_{x \in P} \mathrm{breadth}_P(x).
\]

The objective of this section is the proof of the following theorem. Recall that, given an element $x$ of a join-semilattice $P$ and an ideal $I \subseteq P$, the set $x \vee I$ denotes the join-subsemilattice $\{x \vee y: y \in I\}$.

\begin{theorem}\label{thm:dimension}
Given a join-semilattice $P$ of finite breadth, an ideal $I\subseteq P$, and an element $x \in P$, there exists $y \in P$ with $y \sim_I x$ such that
\[
\mathrm{breadth}(y \vee I) \le \mathrm{breadth}(P) - \mathrm{breadth}_{P/I}([x]_I).
\]
\end{theorem}

Let us start by recalling the following lemma due to Ditor. We omit its proof, since it uses essentially the same argument as the proof of Theorem~\ref{thm:dimension}. The lemma implies the case of Theorem~\ref{thm:dimension} in which $\mathrm{breadth}_{P/I}([x]_I)=1$; in fact, in this case one may take $y=x$.

\begin{lemma}[{\cite[Lemma 5.1]{MR0732199}}]\label{lemma:basicproj}
Given a join-semilattice $P$ of finite breadth, a proper ideal $I \subsetneq P$, and an element $x \in P \setminus I$, $\mathrm{breadth}(x \vee I) < \mathrm{breadth}(P)$.
\end{lemma}
%
%
%

Before delving into the proof of Theorem~\ref{thm:dimension}, let us record the following corollary of Lemma~\ref{lemma:basicproj}, which was already anticipated in the introduction and allows us to reduce the statement of Theorem~\ref{thm:main} to the case $n = 3$.

\begin{corollary}\label{cor:induction}
If there is no lower finite lattice of breadth $3$ and cardinality $\aleph_2$, then, for every $n \ge 3$, there is no lower finite lattice of breadth $n$ and cardinality $\aleph_{n-1}$.
\end{corollary}
\begin{proof}
It suffices to prove that if there is a lower finite lattice of breadth $n > 0$ and cardinality $\aleph_{k}$, for some $k > 0$, then there exists a lower finite lattice of breadth at most $n-1$ and cardinality $\aleph_{k-1}$. The result then follows by a straightforward induction. 

Pick a lower finite lattice $P$ of breadth $n$ and cardinality $\aleph_k$. Pick a (proper) ideal $I \subseteq P$ of cardinality $\aleph_{k-1}$---just let $I$ be the ideal generated by some subset of $P$ of cardinality $\aleph_{k-1}$. Fix also an $x \in P \setminus I$. By Lemma~\ref{lemma:basicproj}, the join-subsemilattice $x \vee I$ has breadth at most $n-1$. Moreover, note that $|x \vee I| = |I|$. Clearly, $|x \vee I| \le |I|$; also, since $I$ is included in the downward closure of $x \vee I$ and $P$ is lower finite, we have $|I| \le |x \vee I| \cdot \aleph_0 = |x \vee I|$. Overall, $x \vee I$ is a lower finite join-semilattice of breadth at most $n-1$ and cardinality $\aleph_{k-1}$. As every lower finite join-semilattice with a least element is a lattice, $x \vee I$ is actually a lattice, and we are done.
\end{proof}

We are ready to prove Theorem~\ref{thm:dimension}.

\begin{proof}[Proof of {\upshape Theorem~\ref{thm:dimension}}]
If $P$ has breadth $0$, then the claim follows trivially. So fix a join-semilattice $P$ of breadth $n > 0$, an ideal $I \subseteq P$, and an element $x \in P$ such that $\mathrm{breadth}_{P/I}([x]_I) = m \in \omega$. We want to show that there is $y \in [x]_I$ such that $\mathrm{breadth}(y \vee I) \le n-m$.

First note that it follows directly from the definition of the quotient join operator that $\mathrm{breadth}(P/I) \le \mathrm{breadth}(P)$, and hence $n \ge m$. Moreover, since $x \vee I$ is a join-subsemilattice of $P$, we have $\mathrm{breadth}(x \vee I) \le \mathrm{breadth}(P)$. In particular, our claim holds when $m = 0$. 

If $m = 1$, then, by letting $y  = x$, our claim follows directly from Lemma~\ref{lemma:basicproj}. Indeed, note that $\mathrm{breadth}_{P/I}([x]_I) > 0$ is equivalent to $x \not\in I$. Hence, from now on, we assume $m > 1$.

Since $\mathrm{breadth}_{P/I}([x]_I) = m$, we can fix some $F \subseteq P/I$ of size $m$ such that $\bigvee F = [x]_I$ and such that for every $H \in [F]^{m-1}$, $\bigvee H < [x]_I$. 

Pick $x_0, \ldots, x_{m-1} \in P$ such that $F = \{[x_i]_I : i < m\}$, and denote $\bigvee_{i < m} x_i$ by $y$. By the way we chose $F$, we must have $x_i \not\in I$ for every $i < m$. Moreover, it directly follows from the quotient map $q_I$ being a join-homomorphism that $[y]_I = \bigvee F = [x]_I$. We claim that $y$ satisfies the desired property, i.e., $\mathrm{breadth}(y \vee I) \le n-m$.

Fix $A \subseteq y \vee I$ of size $n-m+1$ and pick also a set $B \in [I]^{n-m+1}$ such that $y \vee B = A$.

Consider the set $B \cup \{x_0, \ldots, x_{m-1}\}$. Since $B \subseteq I$ and $x_i \not\in I$ for every $i < m$, we conclude that the set $B \cup \{x_0, \ldots, x_{m-1}\}$ has size $n+1$. Since $P$ has breadth $n$, we can pick some $D \in [B \cup \{x_0, \ldots, x_{m-1}\}]^{n}$ such that $\bigvee D = y \vee \bigvee B = \bigvee A$. 

Note that $x_i$ must belong to $D$ for every $i < m$. Otherwise, if there were some $i < m$ such that $x_i\not\in D$, we would conclude that $x_i \le (\bigvee_{j \neq i} x_j) \vee \bigvee B$, and therefore, by applying the quotient map $q_I$ to both sides, we would have $[x_i]_I \le \bigvee (F \setminus \{[x_i]_I\})$, contradicting the way we chose $F$.

Since $D$ has size $n$ and every $x_i$ belongs to it, there must be some $z \in B$ such that $z\not\in D$. Hence 
\begin{equation}\label{eq:dimension}
z \le y \vee \bigvee (B \setminus \{z\}).
\end{equation}

If $n = m$, then $B = \{z\}$ and \eqref{eq:dimension} is equivalent to $z \le y$. In particular, by the arbitrariness of $A$, in this case we conclude that $y$ is an upper bound of $I$, or, equivalently, that $y \vee I = \{y\}$, which proves our claim, as then $\mathrm{breadth}(y \vee I) = 0 = n-m$. 

If $n > m$ instead,  \eqref{eq:dimension} gives us
\[
\bigvee A = y \vee \bigvee B = y \vee \bigvee (B \setminus \{z\}) = \bigvee (y \vee (B \setminus \{z\})).
\]
Therefore, the set $y \vee (B \setminus \{z\})$ is a subset of $A$ of size $n-m$ whose join is the same as the join of $A$. This shows that $\mathrm{breadth}(y \vee I) \le n-m$, and we are done.
\end{proof}

The element $y$ in the statement of Theorem~\ref{thm:dimension} cannot in general be taken to be $x$. Indeed, consider the join-semilattice $P$ represented by the following Hasse diagram, where the dashed ellipse encloses an ideal $I$:\vspace{0.5em}
\usetikzlibrary{fit,shapes.geometric}
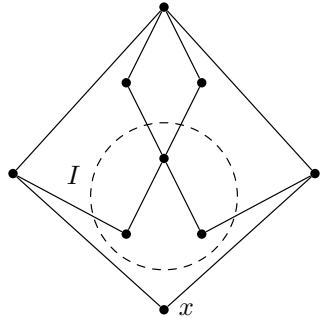
\begin{center}
\begin{tikzpicture}[
    x=1cm,
    y=1cm,
    vertex/.style={circle,fill,inner sep=1.25pt},
    every path/.style={line width=0.45pt,line cap=round}
]
    \node[vertex,label=right:$x$] (x) at (0,0) {};
    \node[vertex] (a) at (-0.5,1) {};
    \node[vertex] (b) at ( 0.5,1) {};

    \node[vertex] (d) at (0,2) {};

    \node[vertex] (e) at (-2,1.8) {};
    \node[vertex] (g) at (-0.5,3) {};
    \node[vertex] (h) at ( 0.5,3) {};
    \node[vertex] (f) at ( 2,1.8) {};

    \node[vertex] (i) at (0,4) {};

 \node[ draw, dashed, ellipse, fit=(a)(b)(d), label={[yshift=8pt]left:$I$} ] {};
    \draw
        (a)--(d)
        (a)--(e)
        (b)--(d)
        (b)--(f)
        (x)--(e)
        (x)--(f)
        (d)--(g)
        (d)--(h)
        (e)--(i)
        (g)--(i)
        (h)--(i)
        (f)--(i);
\end{tikzpicture}\vspace{0.5em}
\end{center}
The join-semilattice $P$ has breadth $3$, while $P/I$ is isomorphic to the four-element Boolean lattice $\mathsf{B}_2$. Moreover, $[x]_I$ is the top element of $P/I$, and hence $P/I$ has breadth $2$ at $[x]_I$. On the other hand, $x\vee I$ consists of two incomparable elements and their join, so it also has breadth $2$. Therefore, 
\[
\mathrm{breadth}(x \vee I) = 2 \not\le \mathrm{breadth}(P) - \mathrm{breadth}_{P/I}([x]_I) = 1.
\]

We now turn to some consequences of Theorem~\ref{thm:dimension}. First, recall that, as already noted in the proof of Theorem~\ref{thm:dimension}, taking quotients cannot increase the breadth: given a congruence relation $\sim$ on a join-semilattice $P$ of finite breadth, the definition of the join operation on $P/{\sim}$ immediately yields  $\mathrm{breadth}(P/{\sim})\leq\mathrm{breadth}(P)$. Theorem~\ref{thm:dimension} shows that this inequality is strict for quotients induced by unbounded ideals.

\begin{corollary}\label{cor:quotient}
Given a join-semilattice $P$ of finite breadth and an unbounded ideal $I\subseteq P$, $\mathrm{breadth}(P/I) < \mathrm{breadth}(P)$.
\end{corollary}
\begin{proof}
For every $x \in P$, we must have $\mathrm{breadth}(x \vee I)>0$; otherwise, $x \vee I$ would be a singleton, and its unique element would be an upper bound of $I$, contrary to the assumption that $I$ is unbounded. Therefore, it follows from Theorem~\ref{thm:dimension} that $\mathrm{breadth}_{P/I}([x]_I) < \mathrm{breadth}(P)$ for every $x \in P$. Since $\mathrm{breadth}(P/I) = \max_{x \in P} \mathrm{breadth}_{P/I}([x]_I)$, we conclude that $\mathrm{breadth}(P/I) < \mathrm{breadth}(P)$.
\end{proof}

Finally, let us prove the following result, obtained by combining Theorem~\ref{thm:dimension} with Ditor's Theorem~\ref{thm:ditor}. It will play an important role in the next section.

\begin{theorem}\label{thm:chain}
Let $\kappa$ be an infinite cardinal, let $n$ be a positive integer, and let $P$ be a join-semilattice of breadth $n$ and cardinality $\kappa^{+(n-1)}$ whose principal ideals have cardinality $< \kappa$. Then, for every ideal $I \subseteq P$ of cardinality at most $\kappa$, there exists $x \in P$ such that $x \vee I$ is a chain.
\end{theorem}
\begin{proof}
If $n = 1$, then the result follows trivially, since $P$ is itself a linear order. So let us assume $n > 1$. For every $x \in P$,
\[
[x]_I \subseteq \bigcup_{y \in I} {\downarrow}(x \vee y).
\] 
Since $I$ has cardinality at most $\kappa$ and the principal ideals of $P$ have cardinality  strictly less than $\kappa$, it follows that every equivalence class of $\sim_I$ must have cardinality at most $\kappa$. Similarly, every principal ideal of $P/I$ has cardinality $\le \kappa$.

As we are assuming $n > 1$, $\kappa^{+(n-1)}$ is a regular cardinal. Since every equivalence class of $\sim_I$ has cardinality at most $\kappa$, $P/I$ has the same cardinality as $P$. Overall,  $P/I$ is a join-semilattice of cardinality $\kappa^{+(n-1)}$ whose principal ideals have cardinality $\le \kappa$. It follows from the first inequality of Ditor's Theorem~\ref{thm:ditor} that $\mathrm{breadth}(P/I) \ge n-1$.  In particular, there must be some $x \in P$ such that $\mathrm{breadth}_{P/I}([x]_I) \ge n-1$. By Theorem~\ref{thm:dimension}, there exists $y \in P$ with $y \sim_I x$ such that $\mathrm{breadth}(y \vee I) \le n - (n-1) = 1$ or, equivalently, such that $y \vee I$ is a chain.
\end{proof}

\section{The filter of projected upper cones}\label{sec:proj}

In this section, we study projected upper cones in lower finite lattices and the filters they generate. Given an infinite lower finite lattice $P$, we associate to each infinite ideal $I\subseteq P$ the following family of subsets of $I$:
\[
\mathcal{F}_I^P \coloneqq \big\{U \subseteq I : \exists x \in P \ (\pi_I [P \uparrow x] \subseteq U)\big\}.
\]

When there is no risk of ambiguity, we just write $\mathcal{F}_I$ instead of $\mathcal{F}_I^P$. This family is the filter on the set $I$ generated by the projections of the upper cones of $P$ onto the ideal $I$.

\begin{lemma}
$\mathcal{F}_I$ is a filter on $I$.
\end{lemma}
\begin{proof}
Clearly, $I \in \mathcal{F}_I$, $\emptyset \not\in \mathcal{F}_I$, and $\mathcal{F}_I$ is closed under taking supersets. Now let us show that $\mathcal{F}_I$ contains the Fréchet filter on $I$. Let $F \subseteq I$ be a finite set. Since $I$ is infinite and $P$ is lower finite, there exists $x \in I$ with $\bigvee F < x$. Then, consider $\pi_I[{\uparrow} x]$. As $x \in I$, it is easy to see that $\pi_I[{\uparrow} x] = I \cap ({\uparrow}x)$. Hence, $F \cap \pi_I[{\uparrow}x] = \emptyset$. In particular, $I \setminus F \supseteq \pi_I[{\uparrow} x]$, and thus $I \setminus F \in \mathcal{F}_I$.

Now pick $U, V \in \mathcal{F}_I$. We show that $U \cap V \in \mathcal{F}_I$. By definition, there are $x,y \in P$ such that $U \supseteq \pi_I[{\uparrow} x]$ and $V \supseteq \pi_I[{\uparrow} y]$. But then, $\pi_I[{\uparrow} (x \vee y)] \subseteq   \pi_I[{\uparrow} x] \cap \pi_I[{\uparrow} y]$. Thus, $U \cap V \supseteq  \pi_I[{\uparrow} x] \cap \pi_I[{\uparrow} y] \supseteq  \pi_I[{\uparrow} (x \vee y)]$ and therefore $U \cap V\in \mathcal{F}_I$.
\end{proof}

We now prove some basic facts about the projections of upper cones.
\begin{lemma}\label{lemma:combproj}
Given a lower finite lattice $P$, an ideal $I \subseteq P$, and an element $x \in P$, the following hold: 
\begin{enumerate}[label={\upshape (\arabic*)}]
\item $\pi_I [{\uparrow}x]$ is a meet-subsemilattice of $P$ which is cofinal in $I$.
\item For every $y \in \pi_I [{\uparrow}x]$, $y = \pi_I(x \vee y)$.
\item The map $\pi_I \upharpoonright (x \vee I)$ is an isomorphism between $x \vee I$ and $\pi_I[{\uparrow} x]$ with their induced orderings.
\end{enumerate}
\end{lemma}
\begin{proof}
That $\pi_I [{\uparrow}x]$ is a meet-subsemilattice of $P$ follows immediately from the fact that $\pi_I$ is a meet-homomorphism by Lemma~\ref{lemma:meetproj}. That it is cofinal in $I$ is also clear: for any $y \in I$, we have $y \le \pi_I(x \vee y) \in \pi_I[{\uparrow} x]$.

Now let us prove (2). Clearly, $y \le \pi_I(x \vee y)$.  Moreover, since $y \in \pi_I [{\uparrow}x]$, there exists $z \ge x$ such that $\pi_I(z) = y$. In particular, since $y \le z$, we have $x \vee y \le z$ and hence $\pi_I(x \vee y) \le \pi_I(z) = y$. Overall,
\[
y \le \pi_I(x \vee y) \le \pi_I(z) = y,
\]
and we conclude $\pi_I(x \vee y) = y$.

We now show (3). The map $\pi_I$ is order-preserving. Moreover, it follows directly from (2) that $\pi_I[x \vee I] = \pi_I[{\uparrow}x]$. Let us prove that $\pi_I \upharpoonright (x \vee I)$ is injective. 

Assume that $\pi_I(x \vee y) = \pi_I(x \vee z)$ with $y,z \in I$. We show that $x \vee y = x \vee z$. The following inequalities hold
\[
x \vee y \le x \vee \pi_I(x \vee y) \le x \vee y,
\]
where the first inequality follows from $y \le \pi_I(x \vee y)$, and the second one from $\pi_I(x \vee y) \le x \vee y$. We conclude that $x \vee y = x \vee \pi_I(x \vee y)$, and similarly $x \vee z = x \vee \pi_I(x \vee z)$. Since $\pi_I(x \vee y) = \pi_I(x \vee z)$ by assumption, we get $x \vee y = x \vee z$. Hence, $\pi_I \upharpoonright (x \vee I)$ is injective. 

What is the inverse of $\pi_I \upharpoonright (x \vee I)$? By (2), it is the map $y \mapsto x \vee y$, which is clearly order-preserving. We conclude that $\pi_I \upharpoonright (x \vee I)$ is an isomorphism between $x \vee I$ and $\pi_I[{\uparrow} x]$ with their induced orderings.
\end{proof}

By Lemma~\ref{lemma:combproj}(3), the meet-subsemilattice $\pi_I[{\uparrow}x]$ and the join-subsemilattice $x\vee I$ are isomorphic with their induced orderings. We can therefore transfer the analysis of the breadth of $x\vee I$ from the previous section to projected upper cones. In particular, Theorem~\ref{thm:chain} and Lemma~\ref{lemma:combproj}(3) yield the following direct corollary.

\begin{theorem}\label{thm:chainprojection}
Let $P$ be a lower finite lattice of breadth $n$ and cardinality $\aleph_{n-1}$ for some $n > 0$. If $I \subseteq P$ is a countably infinite ideal, then there exists $U \in\mathcal{F}_I$ such that $U$ is a chain.
\end{theorem}


Finally, let us prove two consequences of Theorem~\ref{thm:chainprojection}. The first, which is employed in the proof of Theorem~\ref{thm:main}, tells us that, under the relevant hypotheses, the map $I \mapsto \mathcal{F}_I$, which associates to each countably infinite ideal its filter of projected upper cones, is order-preserving with respect to the inclusion relation on ideals and the Rudin-Blass reducibility on filters.  In particular, when $P$ is a lower finite lattice of breadth $3$ and cardinality $\aleph_2$, the following proposition, together with Lemma~\ref{lemma:preservmeager}, implies that if $\mathcal{F}_I$ is meager for some countably infinite ideal $I$, then $\mathcal{F}_J$ is meager for \emph{every} countably infinite ideal $J$. Indeed, given countably infinite ideals $I$ and $J$, the ideal generated by $I\cup J$ is countably infinite and contains both $I$ and $J$. In other words, the meagerness of the filter $\mathcal{F}_I$ is independent of the choice of the countably infinite ideal $I$.

\begin{proposition}\label{prop:rudinblass}
Let $P$ be a lower finite lattice of breadth $n > 0$ and cardinality $\aleph_{n-1}$, and let $I \subseteq J$ be two countably infinite ideals of $P$. Then, $\mathcal{F}_I \le_{RB} \mathcal{F}_J$.
\end{proposition}
\begin{proof}
By Theorem~\ref{thm:chainprojection}, we can pick an $x \in P$ such that $D \coloneqq \pi_J [{\uparrow} x]$ is a chain. The set $D$ belongs to $\mathcal{F}_J$ by definition, and moreover, by Lemma~\ref{lemma:combproj}(1), $D$ is cofinal in $J$. 

Let us first show that $\pi_I \upharpoonright D$ is finite-to-one. Fix $y \in I$. We show that $\pi_I^{-1}(\{y\}) \cap D$ is finite. Since $I$ is an infinite ideal and $P$ is lower finite, there must be $z \in I$ with $y < z$. Moreover, as $D$ is cofinal in $J$  and $z  \in I \subseteq J$, there must be $u \in D$ with $z \le u$. Clearly, $\pi_I [{\uparrow} u] \subseteq {\uparrow} z$, and therefore $y \not\in \pi_I [{\uparrow} u]$. Furthermore, since $D$ is a chain, we have
\[
D = (D \cap ({\uparrow} u)) \cup (D \cap ({\downarrow} u)).
\] 
We conclude that $\pi_I^{-1}(\{y\}) \cap D \subseteq {\downarrow} u$, which is finite. Hence, $\pi_I \upharpoonright D$ is finite-to-one.

Since $D$ is an infinite subset of the countable set $J$, we can fix a finite-to-one map $f:J \rightarrow D$ such that $f \upharpoonright D$ is the identity. Let $\theta \coloneqq \pi_I \circ f$. Since both $f$ and $\pi_I \upharpoonright D$ are finite-to-one, so is $\theta$. Now pick some $U \subseteq I$ and let us show that $U \in \mathcal{F}_I$ if and only if $\theta^{-1}(U) \in \mathcal{F}_J$.

First suppose that $U \in \mathcal{F}_I$. We show that $\theta^{-1}(U) \in \mathcal{F}_J$. By definition, there exists some $y \in P$ such that $U \supseteq \pi_I[{\uparrow} y]$. Then, the following holds:
\begin{align*}
\theta^{-1}(U) &\supseteq \theta^{-1}(\pi_I[{\uparrow} y])\\
&= f^{-1}(\pi_I^{-1}(\pi_I [{\uparrow} y]))\\
&\supseteq f^{-1}(\pi_I^{-1}(\pi_I [{\uparrow} y]) \cap D)\\
&\supseteq \pi_I^{-1}(\pi_I [{\uparrow} y]) \cap D\\
&= \pi_I^{-1}(\pi_I \circ \pi_J[{\uparrow} y]) \cap D\\
&\supseteq \pi_J[{\uparrow} y] \cap D \in \mathcal{F}_J,
\end{align*}
where the third inclusion follows from $f \upharpoonright D$ being the identity and the second equality follows from Lemma~\ref{lemma:concproj}. In particular, we conclude that $\theta^{-1}(U) \in \mathcal{F}_J$.

Now suppose that $\theta^{-1}(U) \in \mathcal{F}_J$. We show that $U \in \mathcal{F}_I$. By definition, there must be some $y \in P$ such that $\theta^{-1}(U) \supseteq \pi_J[{\uparrow} y]$. Replacing $y$ by $x \vee y$ if necessary, we may assume that $ \pi_J[{\uparrow} y] \subseteq D$. Then,
\begin{align*}
U &\supseteq \theta \circ \pi_J[{\uparrow} y]\\
&= \pi_I \circ f \circ \pi_J[{\uparrow} y]\\
&= \pi_I \circ \pi_J[{\uparrow} y]\\
&= \pi_I[{\uparrow} y] \in \mathcal{F}_I,
\end{align*}
where the second equality follows from $f \upharpoonright D$ being the identity and $D \supseteq  \pi_J[{\uparrow} y]$, while the last equality follows from Lemma~\ref{lemma:concproj}. In particular, we conclude that $U \in \mathcal{F}_I$. Overall, $\theta$ witnesses $\mathcal{F}_I \le_{RB} \mathcal{F}_J$.
\end{proof}

The second and final result of this section shows that, under the relevant cardinality hypotheses, the filters of projected upper cones on countably infinite ideals are all $\mathsf{P}$-filters. Recall that a filter $\mathcal{F}$ is a $\mathsf{P}$-filter if, for every countable collection $\{U_n : n\in\omega\} \subseteq \mathcal{F}$, there exists $U \in \mathcal{F}$ such that $U_n \supseteq^* U$ for every $n$. Although the next result establishes an important combinatorial feature of these filters, it will not be employed in the proof of Theorem~\ref{thm:main}.

\begin{proposition}
Let $P$ be a lower finite lattice of breadth $n > 0$ and cardinality $\aleph_{n-1}$. Then, $\mathcal{F}_I$ is a $\mathsf{P}$-filter for every countably infinite ideal $I \subseteq P$.
\end{proposition}
\begin{proof}
If $n = 1$, note that $P$ must be isomorphic to $\omega$ with its usual ordering. The only infinite ideal of $\omega$ is $\omega$ itself, and $\mathcal{F}_\omega = \mathfrak{Fr}$, which is trivially a $\mathsf{P}$-filter. So we can assume $n > 1$.

Pick a countably infinite ideal $I \subseteq P$ and a countable family $\{U_k : k\in\omega\} \subseteq \mathcal{F}_I$. By definition of $\mathcal{F}_I$, for each $k$ we can pick $x_k \in P$ such that $U_k \supseteq \pi_I[{\uparrow} x_k]$. Arguing as in the proof of Theorem~\ref{thm:chain}, we have that the quotient $P/I$ is a join-semilattice of cardinality $\aleph_{n-1}$ whose principal ideals are countable. Moreover, as $P$ is lower finite and $I$ is infinite, $I$ is unbounded, and we conclude by Corollary~\ref{cor:quotient} that $P/I$ has breadth at most $n-1$. 

By \cite[Proposition 11]{MR4993406}, every join-semilattice of breadth (at most) $n-1$ and cardinality $\aleph_{n-1}$ whose principal ideals are countable is $\sigma$-directed---i.e., every countable subset has an upper bound. Thus, $P/I$ is $\sigma$-directed, and we can pick $y \in P$ such that $[x_k]_I \le [y]_I$ for every $k$. 

By Theorem~\ref{thm:chainprojection}, there exists $x \in P$ such that $\pi_I[{\uparrow} x]$ is a chain. Consider the element $x \vee y$. We claim that $\pi_I[{\uparrow} (x \vee y)] \subseteq^* U_k$ for every $k$. Since $\pi_I[{\uparrow} (x \vee y)]$ belongs to $\mathcal{F}_I$, this claim will complete the proof.

Fix $k\in\omega$. We show that $\pi_I[{\uparrow} (x \vee y)] \subseteq^* U_k$. As $[x_k]_I \le [y]_I$, there exists $a \in I$ such that $x_k \le a \vee y$. Thus, 
\[
({\uparrow}a) \cap \pi_I[{\uparrow} y] = \pi_I[{\uparrow} (a \vee y)] \subseteq \pi_I[{\uparrow} x_k] \subseteq U_k.
\]

By Lemma~\ref{lemma:combproj}(1), $\pi_I[{\uparrow} (x \vee y)]$ is cofinal in $I$. Pick $u \in \pi_I[{\uparrow} (x \vee y)]$ with $a \le u$. Since $\pi_I[{\uparrow} (x \vee y)] \subseteq \pi_I[{\uparrow} x]$, the set $\pi_I[{\uparrow} (x \vee y)]$ is also a chain. Thus,
\[
\pi_I[{\uparrow} (x \vee y)] \setminus ({\uparrow} a) \subseteq \pi_I[{\uparrow} (x \vee y)] \cap ({\downarrow} u),
\] 
which is finite, as $P$ is lower finite. Overall, we have
\[
 \pi_I[{\uparrow} (x \vee y)] =^* ({\uparrow}a) \cap \pi_I[{\uparrow} (x \vee y)] \subseteq ({\uparrow}a) \cap \pi_I[{\uparrow} y] \subseteq U_k,
\] 
and we conclude $\pi_I[{\uparrow} (x \vee y)] \subseteq^* U_k$.
\end{proof}

\section{The main result}\label{sec:main}

This section is devoted to the proof of the following theorem. By combining it with Corollary~\ref{cor:induction}, we get Theorem~\ref{thm:main}. We refer the reader to Section~\ref{sec:prel:set} for the relevant set-theoretic terminology and for the basics of Lévy's collapse forcing $\mathrm{Coll}(\omega_1, {<}\kappa)$. Familiarity with standard forcing terminology and techniques is assumed throughout the section.

\begin{theorem}\label{thm:maind}
If $\kappa$ is a Mahlo cardinal and $G$ is a $\mathsf{V}$-generic filter for $\mathrm{Coll}(\omega_1, {<}\kappa)$, then in $\mathsf{V}[G]$ there are no lower finite lattices of breadth $3$ and cardinality $\aleph_2$.
\end{theorem}

First, let us prove the following technical proposition.

\begin{proposition}\label{prop:covering}
Let $P$ be a lower finite lattice of cardinality $\ge \aleph_2$ and let $\mathbb{P}$ be a forcing notion that preserves $\omega_1$. If $\dot{C}$ is a $\mathbb{P}$-name such that 
\begin{multline*}
\Vdash \dot{C} \text{ is a cofinal subset of } P \text{ and } \dot{C} \text{ is a join-semilattice of breadth } 2 \\\text{in the induced ordering},
\end{multline*}
then, for every $p \in \mathbb{P}$, there is $x \in P$ such that for every $F \in [{\uparrow} x]^{\le \aleph_1}$, $p \not\Vdash \dot{C} \cap F \neq \emptyset$.
\end{proposition}

The proof uses the following auxiliary lemma.

\begin{lemma}\label{lemma:covering}
Let $\mathsf{V \subseteq W}$ be transitive class models of  $\mathsf{ZFC}$ with $\omega_1^{\mathsf{V}} = \omega_1^\mathsf{W}$, and let $P \in \mathsf{V}$ be an uncountable lower finite lattice. Suppose that $C \in \mathsf{W}$ is a cofinal subset of $P$ and $C$ is a join-semilattice of breadth $2$ in the induced ordering. Then, for every uncountable proper ideal $I  \in \mathsf{V}$ of $P$, $C \cap I$ is not cofinal in $I$.
\end{lemma}
\begin{proof}
Suppose, toward a contradiction, that there is an uncountable proper ideal $I \in \mathsf{V}$ of $P$ such that $I \cap C$ is cofinal in $I$. In particular, $C \cap I$ is a directed subset of $C$---i.e., any two elements of $C \cap I$ have an upper bound in $C \cap I$. Moreover, as $I$ is downward closed, $C \cap I$ is also a downward closed subset of $C$. Overall, $C \cap I$, being a downward closed and directed subset of $C$, is an ideal of $C$ with respect to the induced ordering.

Actually, $C \cap I$ is a \emph{proper} ideal of $C$. Indeed, $I$ is  a proper ideal of $P$, and therefore $C \setminus I \neq \emptyset$, since $C$ is a cofinal subset of $P$.

By the second inequality of Ditor's Theorem~\ref{thm:ditor}, every proper ideal of a lower finite join-semilattice of breadth $2$ must be countable. Hence, by the preceding paragraph, $C \cap I$ must be countable in $\mathsf{W}$. Now, since $C \cap I$ is cofinal in $I$, we must have $|C \cap I|^\mathsf{W} = |I|^\mathsf{W}$---in general, every cofinal subset of a lower finite infinite poset must have the same cardinality as the whole poset. Thus, $|I|^\mathsf{W} = \aleph_0$, which is a contradiction, as $|I|^\mathsf{V} \ge \omega_1^\mathsf{V} = \omega_1^\mathsf{W}$. 
\end{proof}

\begin{proof}[Proof of \upshape Proposition~\ref{prop:covering}]
Suppose, toward a contradiction, that there exists a $p \in \mathbb{P}$ such that for every $x \in P$, there is an $F \in [P]^{\le \aleph_1}$ with $F \subseteq {\uparrow}x$ and $p \Vdash \dot{C} \cap F \neq \emptyset$.

Let $M \prec H(\theta)$ be an elementary submodel, where $\theta$ is a sufficiently large regular cardinal, such that $|M| = \aleph_1$, $\omega_1 \subseteq M$ and $P, \mathbb{P}, p, \dot{C} \in M$. Let $I = P \cap M$. It easily follows from $P$ being lower finite and from the elementarity of $M$ that $I$ is a proper ideal of $P$ of cardinality $\aleph_1$. We claim that
\begin{equation}\label{eq:covering}
p \Vdash \dot{C} \cap I \text{ is cofinal in }I.
\end{equation}
Pick some $x \in I$. By hypothesis, there exists $F \in [P]^{\le \aleph_1}$ with $F \subseteq {\uparrow}x$ and $p \Vdash \dot{C} \cap F \neq \emptyset$. By elementarity of $M$, such an $F$ can be found in $M$. But from $F \in M$ and $\omega_1 \subseteq M$, it follows that $F \subseteq I$. In particular, $p \Vdash ({\uparrow}x) \cap I \cap \dot{C} \neq \emptyset$. Hence, \eqref{eq:covering} holds. 

Let $G$ be a $\mathsf{V}$-generic filter for $\mathbb{P}$ with $p \in G$. Then, $\dot{C}_{G}$---i.e., the interpretation of the name $\dot{C}$ by $G$---is a cofinal subset of $P$ that, in its induced ordering, is a join-semilattice of breadth $2$.

Because $\mathbb{P}$ preserves $\omega_1$, we have $\omega_1^{\mathsf{V}} = \omega_1^{\mathsf{V}[G]}$. Thus, all the hypotheses of Lemma~\ref{lemma:covering} are satisfied by the models $\mathsf{V} \subseteq \mathsf{V}[G]$, the lattice $P$, the ideal $I$, and the set $\dot{C}_G$. Consequently, $\dot{C}_G \cap I$ is not cofinal in $I$, contradicting \eqref{eq:covering}.
\end{proof}

We now prove Theorem~\ref{thm:maind} by contradiction. Suppose that $\kappa$ is a Mahlo cardinal, $G$ is a $\mathsf{V}$-generic filter for $\mathrm{Coll}(\omega_1,{<}\kappa)$, and in $\mathsf{V}[G]$ there exists a lower finite lattice $P$ of breadth $3$ and cardinality $\aleph_2$. The remainder of this section is devoted to deriving a contradiction from this assumption.

Since $\kappa=\omega_2^{\mathsf{V}[G]}$, we may assume without loss of generality that $P=(\kappa,\preceq)$, and, by permuting $\kappa$ if necessary, that $\omega$ is an ideal of $P$. We use $\preceq$, rather than $\le$, to distinguish this lattice ordering from the usual well-ordering of the ordinal $\kappa$. We begin by proving that in $\mathsf{V}[G]$ the filter of projected upper cones onto $\omega$ is non-meager.

\begin{proposition}\label{prop:nonmeager}
In $\mathsf{V}[G]$, $\mathcal{F}_\omega^P$ is non-meager.
\end{proposition}
\begin{proof}
For each $\lambda < \kappa$, let $G \upharpoonright \lambda \coloneqq G \cap \mathrm{Coll}(\omega_1, {<}\lambda)$. Since $\mathrm{Coll}(\omega_1, {<}\lambda)$ is a complete subforcing of $\mathrm{Coll}(\omega_1, {<}\kappa)$, $G \upharpoonright \lambda$ is $\mathsf{V}$-generic for $\mathrm{Coll}(\omega_1, {<}\lambda)$. 

In the ground model $\mathsf{V}$, let $\dot{\preceq}$ be a $\mathrm{Coll}(\omega_1, {<}\kappa)$-name for the partial order $\preceq$ on $\kappa$, let $\dot{G}$ be the canonical name for the generic filter, and pick $p \in G$ such that
\[
p \Vdash (\kappa, \dot{\preceq}) \text{ is a lower finite lattice}.
\]
Then, by a routine argument using the regularity of $\kappa$ and the fact that $\mathrm{Coll}(\omega_1, {<}\kappa)$ is $\kappa$-cc, it follows that the set of ordinals $\lambda < \kappa$ such that 
\begin{equation}\label{eq:initial}
p \Vdash \lambda \text{ is an ideal of } (\kappa, \dot{\preceq}) \text{ and } {\dot{\preceq}} \cap (\lambda \times \lambda) \in \mathsf{V}[\dot{G} \upharpoonright \lambda],
\end{equation}
is a club subset of $\kappa$. Since $\kappa$ is Mahlo in the ground model $\mathsf{V}$, we conclude that we can pick an inaccessible (in $\mathsf{V}$) cardinal $\lambda < \kappa$ satisfying \eqref{eq:initial}. Thus, as $p \in G$, we have that $\lambda$ is an ideal of $P$ (in $\mathsf{V}[G]$) and ${\preceq} \cap (\lambda \times \lambda) \in \mathsf{V}[G \upharpoonright \lambda]$.

 Since $\lambda$ is inaccessible, $\mathrm{Coll}(\omega_1, {<}\lambda)$ preserves $\lambda$, and therefore $\lambda = \omega_2^{\mathsf{V}[G \upharpoonright \lambda]}$. In particular, $(\lambda, \preceq)$ is a lower finite lattice of breadth $3$ and cardinality $\aleph_2$ in $\mathsf{V}[G \upharpoonright \lambda]$. Now fix some $\tilde{
x} \in \kappa \setminus \lambda$.  

\begin{claim}\label{thm:main:claim1}
In $\mathsf{V}[G]$, for every sequence $\langle A_k : k\in\omega\rangle$ of finite pairwise disjoint subsets of $\omega$, there exist a countable ideal $I$ of $P$ and  a set $U\in \mathcal{F}_I^P$ such that:
\begin{enumerate}[label={\upshape (\arabic*)}]
\item $\omega \subseteq I \subseteq \lambda$,
\item $\pi_\lambda[P \uparrow \tilde{x}] \cap I$ is cofinal in $I$,
\item $\pi_\omega [\pi_{\lambda}[P \uparrow \tilde{x}] \cap U] \cap A_k = \emptyset$ for infinitely many $k \in \omega$.
\end{enumerate}
\end{claim}
\begin{why}
For notational clarity, let us denote the projected upper cone $\pi_{\lambda}[ P \uparrow \tilde{x}]$ by $C$. Note that by Lemmas~\ref{lemma:combproj} and \ref{lemma:basicproj}, $C$ is a cofinal meet-subsemilattice of $(\lambda, \preceq)$ of breadth $2$ in the induced ordering.

Fix a sequence $\langle A_k : k\in\omega\rangle$ of finite, pairwise disjoint subsets of $\omega$. We will construct a countable ideal $I$ of $P$ and some $\tilde{y} \in \lambda$ such that $I$ and $U = \pi_I[P \uparrow \tilde{y}] \in \mathcal{F}_I^P$ satisfy (1)-(3). Since we mostly work with the lattice $(\lambda, \preceq)$ until the end of the proof of Claim~\ref{thm:main:claim1}, let us denote it by $Q$ to avoid confusion with the larger lattice $P =  (\kappa, \preceq)$. 

Observe that for every ideal $I$ of $P$ with $I \subseteq \lambda$ and every $y \in \lambda$, $I$ is also an ideal of $Q$ and $\pi_I[P \uparrow y] = \pi_I[Q \uparrow y]$. Indeed, we have
\[
\pi_I[P \uparrow y] = \pi_I[y \vee I] = \pi_I[Q \uparrow y],
\]
where the first equality follows from Lemma~\ref{lemma:combproj}(2) and the second one from the same lemma and $y\vee I \subseteq \lambda$. Consequently, the projected upper cone $\pi_I[P \uparrow y]$ can be computed in $\mathsf{V}[G \upharpoonright \lambda]$ as $\pi_I[Q \uparrow y]$, even though $P$ does not belong to $\mathsf{V}[G \upharpoonright \lambda]$.

Since our collapse forcing is $\sigma$-closed, the sequence $\langle A_k : k\in\omega\rangle$ already belongs to the ground model $\mathsf{V}$. In particular, our sequence belongs to $\mathsf{V}[G \upharpoonright \lambda]$. 

The tail forcing $\mathrm{Coll}(\omega_1, {<}\kappa)/(G \upharpoonright \lambda)$ is equivalent to $\mathrm{Coll}(\omega_1, {<}\kappa)$ in $\mathsf{V}[G \upharpoonright \lambda]$. Hence, there exists a $\mathsf{V}[G \upharpoonright \lambda]$-generic filter $H$ for $\mathrm{Coll}(\omega_1, {<}\kappa)^{\mathsf{V}[G \upharpoonright \lambda]}$ such that $\mathsf{V}[G] = \mathsf{V}[G \upharpoonright \lambda][H]$. From now until the end of the proof of Claim~\ref{thm:main:claim1}, we work in $\mathsf{V}[G \upharpoonright \lambda]$ with the lattice $Q$, and we simply write ${\uparrow} z$ instead of $Q \uparrow z$ for every $z \in \lambda$. 

Fix a $\mathrm{Coll}(\omega_1, {<}\kappa)$-name $\dot{C}$ for $C$ and a condition $p \in H$ such that 
\[
p \Vdash \dot{C} \text{ is a cofinal meet-subsemilattice of }Q \text{ of breadth }2 \text{ in the induced ordering}.
\]

We prove our claim using a density argument. More precisely, we show that for every $q \le p$ there exist $r \le q$, a countable ideal $I$ of $Q$ with  $\omega \subseteq I$, and some $\tilde{y} \in Q$ such that $r$ forces $\dot{C} \cap I$ to be cofinal in $I$ and also forces $\pi_\omega [\dot{C} \cap \pi_I[{\uparrow} \tilde{y}]] \cap A_k = \emptyset$ for infinitely many $k \in \omega$.

Fix any $q \le p$ and a countable elementary submodel $M \prec H(\theta)$ for some sufficiently large regular cardinal $\theta$ with $\kappa, q, \dot{C}, Q \in M$. Denote $\lambda \cap M$ by $I$. It quickly follows from $\preceq$ being lower finite and from the elementarity of $M$ that $I$ is a countably infinite ideal of $Q$ and $\omega \subseteq I$. 

Since $Q$ is a lower finite lattice of breadth $3$ and cardinality $\aleph_2$, it follows from Theorem~\ref{thm:chainprojection} that there exists $\tilde{y} \in Q$ such that $\pi_I[{\uparrow} \tilde{y}]$ is a chain.  Fix one such $\tilde{y}$ and let us denote $\pi_I[{\uparrow} \tilde{y}]$ by $U$. The rest of the proof consists of constructing a condition $r \le q$ which forces $\dot{C} \cap I$ to be cofinal in $I$ and $\pi_\omega[\dot{C} \cap U] \cap A_k = \emptyset$ for infinitely many $k \in\omega$.

For each $k\in\omega$, let $B_k \coloneqq U \cap \pi_{\omega}^{-1}(A_k)$. Since $U$ is a chain cofinal in $I$, the map $\pi_\omega \upharpoonright U$ is finite-to-one (see the beginning of the proof of Proposition~\ref{prop:rudinblass}). Thus, the sets $B_k$ are finite and pairwise disjoint.

Fix an enumeration $\langle x_n : n \in\omega\rangle$ of $I$. We now define a decreasing sequence $\langle q_n : n\in\omega \rangle$ of conditions in $\mathrm{Coll}(\omega_1, {<}\kappa) \cap M$ with $q_0 = q$ and an increasing sequence of natural numbers $\langle k_n : n\in\omega\rangle$  such that, for every $n\in\omega$:
\begin{enumerate}[label=(\alph*)]
\item $q_{n+1} \Vdash \dot{C} \cap B_{k_{n+1}} = \emptyset$, and
\item $q_{n+1} \Vdash \dot{C} \cap I \cap ({\uparrow} x_n) \neq \emptyset$.
\end{enumerate}

First let $q_0 = q$ and $k_0 = 0$. Now suppose that $q_n$ and $k_n$ have been defined. We define $q_{n+1}$ and $k_{n+1}$ as follows. Since $Q$ has cardinality $\aleph_2$, it follows from Proposition~\ref{prop:covering} that there exists $y \in Q$ such that for every finite $F \subseteq {\uparrow} y$, $q_{n} \not\Vdash \dot{C} \cap F \neq \emptyset$. By elementarity of $M$, we can find such a $y$ in $I$.

Since $U$ is a chain cofinal in $I$, the set $U \setminus ({\uparrow}y)$ is finite. 
As the sets $B_k$ are pairwise disjoint subsets of $U$, there exists $k > k_n$ such that $B_k \subseteq {\uparrow} y$. Let $k_{n+1}$ be one such $k$. By the way we chose $y$, we conclude $q_n \not\Vdash \dot{C} \cap B_{k_{n+1}} \neq \emptyset$. By elementarity of $M$, there exists a condition $s \in M$ with $s \le q_n$ such that $s \Vdash \dot{C} \cap B_{k_{n+1}} = \emptyset$---note that $B_{k_{n+1}}$, being a finite subset of $M$, belongs to $M$. Moreover, since $p$ forces $\dot{C}$ to be cofinal in $Q$ and $s$ extends $p$, we can pick $q_{n+1} \in M$ with $q_{n+1} \le s$ such that, for some $z \in I$ with $x_n \preceq z$, $q_{n+1} \Vdash z \in \dot{C}$. This ends the inductive definition of the $q_n$'s and $k_n$'s.

Let $r = \bigcup_n q_n$. By construction, $r$ forces $\dot{C} \cap I$ to be cofinal in $I$, and also forces $\pi_\omega[\dot{C} \cap U] \cap A_{k_n} = \pi_\omega[\dot{C} \cap B_{k_n}] = \emptyset$ for every $n > 0$. We are done.
\end{why}

Now we complete the proof of Proposition~\ref{prop:nonmeager}. By Lemma~\ref{lemma:characterizationmeager}, we need to show that for every sequence $\langle A_k : k\in\omega\rangle$ of finite pairwise disjoint subsets of $\omega$, there exists some $V \in \mathcal{F}_\omega^P$ such that $V \cap A_k = \emptyset$ for infinitely many $k$.

Fix a sequence $\langle A_k : k\in\omega\rangle$ of finite pairwise disjoint subsets of $\omega$. Fix also a countable ideal $I$ of $P$ and some $U \in \mathcal{F}_I^P$ satisfying (1)-(3) of Claim~\ref{thm:main:claim1} for $\langle A_k : k\in\omega\rangle$. Pick $y \in P$ such that $\pi_I[P \uparrow y] \subseteq U$ and let $V = \pi_\omega[P \uparrow (\tilde{x} \vee y)]$. Clearly, $V \in \mathcal{F}_\omega^P$.
\begin{claim}
$V \cap A_k = \emptyset$ for infinitely many $k\in\omega$. 
\end{claim}
\begin{why}
By the way we chose $I$ and $U$, there are infinitely many $k \in\omega$ such that $\pi_\omega [\pi_{\lambda}[ P \uparrow \tilde{x}] \cap U] \cap A_k = \emptyset$. Fix one such $k$. We show that $V \cap A_k = \emptyset$.

We first claim that $\pi_I[P \uparrow \tilde{x}] \subseteq \pi_{\lambda}[P \uparrow\tilde{x}]$.  Pick some $w \in P$ with $\tilde{x} \preceq w$. We show that $\pi_I(w) \in \pi_{\lambda}[P \uparrow\tilde{x}]$. By property (2) of $I$, $\pi_\lambda[P \uparrow \tilde{x}] \cap I$ is cofinal in $I$. Hence, we now fix some $b \in \pi_{\lambda}[P \uparrow\tilde{x}] \cap I$ such that $\pi_I(w) \preceq b$. The following holds:
\begin{align*}
\pi_\lambda(w) \wedge b &= \pi_I(\pi_\lambda(w) \wedge b)\\
&= (\pi_I \circ \pi_\lambda) (w) \wedge b\\
&=\pi_I(w) \wedge b = \pi_I(w),
\end{align*}
where the first equality follows from $\pi_I \upharpoonright I$ being the identity and from $\pi_\lambda(w) \wedge b \in I$, as $b \in I$ and $I$ is downward closed; the second one follows from Lemma~\ref{lemma:meetproj}; the third one follows from $I \subseteq \lambda$ and  Lemma~\ref{lemma:concproj}, and the last one from $\pi_I(w) \preceq b$. Thus, $\pi_\lambda(w) \wedge b = \pi_I(w)$. By Lemma~\ref{lemma:combproj}(1), $\pi_{\lambda}[P \uparrow\tilde{x}]$ is a meet-subsemilattice of $P$, and therefore $\pi_{\lambda}(w) \wedge b \in \pi_{\lambda}[P \uparrow\tilde{x}]$, given that both $\pi_{\lambda}(w)$ and $b$ belong to $\pi_{\lambda}[P \uparrow\tilde{x}]$. We conclude that $\pi_I(w) \in \pi_{\lambda}[P \uparrow \tilde{x}]$. In particular, we have $\pi_I[P \uparrow \tilde{x}] \subseteq \pi_{\lambda}[P \uparrow\tilde{x}]$, as we wanted to show.

It follows from the preceding paragraph and from the way we chose $y$ that 
\[
\pi_I[P \uparrow(\tilde{x} \vee y)] \subseteq \pi_{I}[P \uparrow\tilde{x}] \cap \pi_I[P \uparrow y] \subseteq \pi_{\lambda}[P \uparrow\tilde{x}] \cap U.
\] 
From $\omega \subseteq I$ and from Lemma~\ref{lemma:concproj}, it follows that $\pi_\omega = \pi_\omega \circ \pi_I$. In particular,  
\[
V = \pi_\omega[P \uparrow (\tilde{x} \vee y)] = (\pi_\omega \circ \pi_I)[P \uparrow (\tilde{x} \vee y)] \subseteq \pi_\omega[\pi_{\lambda}[P \uparrow\tilde{x}] \cap U].
\]
Therefore, by the way we chose $k$, we have $V \cap A_k = \emptyset$.  Since there are infinitely many $k$ such that $\pi_\omega[\pi_{\lambda}[P \uparrow\tilde{x}] \cap U] \cap A_k = \emptyset$, our claim follows.
\end{why}
This finishes the proof.
\end{proof}

So, by Proposition~\ref{prop:nonmeager}, in $\mathsf{V}[G]$ the filter $\mathcal{F}_\omega^P$ is non-meager, or, equivalently, for every sequence $\langle A_k : k\in\omega\rangle$ of finite, pairwise disjoint subsets of $\omega$, there exists $x \in \kappa$ such that $\pi_\omega[P \uparrow x] \cap A_k = \emptyset$ for infinitely many $k$. As $\mathsf{CH}$ holds in $\mathsf{V}[G]$ and $\kappa={\omega_2}^{\mathsf{V}[G]}$, there exists $\alpha < \kappa$ such that for every $\langle A_k : k\in\omega\rangle$ of finite, pairwise disjoint subsets of $\omega$, there exists $x \in \alpha$ such that $\pi_\omega[ P \uparrow x] \cap A_k = \emptyset$ for infinitely many $k$.

Arguing as at the beginning of the proof of Proposition~\ref{prop:nonmeager}, we can show that there is a $\mu < \kappa$ with $\alpha \le \mu$ such that $\mu$ is inaccessible in $\mathsf{V}$, $\mu$ is an ideal of $P$ (in $\mathsf{V}[G]$) and ${\preceq} \cap (\mu \times \mu) \in \mathsf{V}[G \upharpoonright \mu]$. Denote by $R$ the lattice $(\mu, \preceq)$. 

\begin{lemma}\label{lemma:abs}
In $\mathsf{V}[G \upharpoonright \mu]$, the filter $\mathcal{F}_\omega^R$ is non-meager.
\end{lemma}
\begin{proof}
Pick a sequence $\langle A_k : k\in\omega\rangle$ in $\mathsf{V}[G \upharpoonright \mu]$ of finite, pairwise disjoint subsets of $\omega$. By the way we chose $\alpha$, we know that, in $\mathsf{V}[G]$, there exists $y \in \alpha$ such that $\pi_\omega[P \uparrow y] \cap A_k = \emptyset$ for infinitely many $k$. Since $\alpha \le \mu$,  $y \in \mu$. Moreover, $R \uparrow y = (P \uparrow y) \cap \mu$. In particular, we conclude that $\pi_\omega[R \uparrow y] \cap A_k = \emptyset$ for infinitely many $k$. Furthermore, the statement ``$\pi_\omega[R \uparrow y] \cap A_k = \emptyset$" is absolute in the relevant parameters, and therefore also in $\mathsf{V}[G \upharpoonright \mu]$ we have $\pi_\omega[R \uparrow y] \cap A_k = \emptyset$ for infinitely many $k$. We conclude that the filter $\mathcal{F}_\omega^R$ is non-meager in $\mathsf{V}[G \upharpoonright \mu]$.
\end{proof}

As we noted at the beginning of the proof of Claim~\ref{thm:main:claim1}, there exists a $\mathsf{V}[G \upharpoonright \mu]$-generic filter $H$ for $\mathrm{Coll}(\omega_1, {<}\kappa)^{\mathsf{V}[G \upharpoonright \mu]}$ such that $\mathsf{V}[G] = \mathsf{V}[G \upharpoonright \mu][H]$.

Fix any $\tilde{x} \in \kappa \setminus \mu$ and let $C \coloneqq \pi_{\mu}[P \uparrow \tilde{x}]$. By Lemma~\ref{lemma:combproj}, Lemma~\ref{lemma:basicproj}, and Ditor's Theorem~\ref{thm:ditor}, $C$ is a cofinal meet-subsemilattice of $R$ of breadth $2$ in the induced ordering. Fix a $\mathrm{Coll}(\omega_1, {<}\kappa)$-name $\dot{C}$ for $C$ and a condition $p \in H$ such that 
\begin{multline}\label{eq:contr}
p \Vdash \dot{C} \text{ is a cofinal meet-subsemilattice of }R \text{ of breadth }2\\\text{in the induced ordering}.
\end{multline}

From now until the end of the section---that is, until the end of the proof of Theorem~\ref{thm:maind}---we work in $\mathsf{V}[G \upharpoonright \mu]$ with the lattice $R$. To simplify the notation, let us denote the filters $\mathcal{F}^R_I$ simply by $\mathcal{F}_I$ and let us write ${\uparrow} z$ instead of $R \uparrow z$. 

 Fix a countable elementary submodel $M \prec H(\theta)$ for some large enough regular cardinal $\theta$ with $\kappa, p, \dot{C}, R  \in M$. Set $I \coloneqq \mu \cap M$. It follows quickly from the lower finiteness of $R$ and from the elementarity of $M$ that $I$ is a countably infinite ideal of $R$ and $\omega \subseteq I$. By Proposition~\ref{prop:rudinblass} and Lemma~\ref{lemma:preservmeager}, it follows from $\mathcal{F}_\omega$ being non-meager that $\mathcal{F}_I$ is also non-meager.

Since $R$ is a lower finite lattice of breadth $3$ and cardinality $\aleph_2$, we can fix a chain $U \in \mathcal{F}_I$ by Theorem~\ref{thm:chainprojection}. Let $\langle u_n : n \in\omega\rangle$ be a $\preceq$-increasing enumeration of $U$. Since $\mathcal{F}_I$ is non-meager and $U \in \mathcal{F}_I$, it is easy to see that the filter 
\[
\tilde{\mathcal{F}}_I \coloneqq \big\{X \in [\omega]^\omega : \{u_n : n \in X\} \in \mathcal{F}_I\big\}
\]  
on $\omega$ is also non-meager---the filters $\mathcal{F}_I$ and $\tilde{\mathcal{F}}_I$ are actually isomorphic.

Now consider the following game $\mathcal{G}^\star$. In this game, at the $k$-th round, Player I plays a condition $p_k$ in $\mathrm{Coll}(\omega_1, {<}\kappa) \cap M$, and then Player II plays another condition $q_k$ in $\mathrm{Coll}(\omega_1, {<}\kappa) \cap M$
\begin{center}
\begin{tabular}{cccccccc}
I & $p_0$ & & $p_1$ & & $p_2$ & & ... \\
II & & $q_0$ & & $q_1$ & & $q_2$ & ...\\
\end{tabular}
\end{center}
\noindent with the rule: $p_{k+1} \le q_k \le p_k \le p$ for every $k \in \omega$. At the end of a play, Player II \emph{wins} if
\[
\bigcap_{k\in\omega} \{x \in I : p_k \not\Vdash x \not\in \dot{C}\} \in \mathcal{F}_I^+.
\]
In what follows, $\mathcal{G}$ refers to Laflamme's game discussed in Section~\ref{sec:prel:filters}. Strategies for the game $\mathcal{G}^\star$ are defined analogously to those for Laflamme's game.

\begin{lemma}\label{lemma:game}
If Player II has a winning strategy in $\mathcal{G}^\star$, then Player II has a winning strategy in $\mathcal{G}(\tilde{\mathcal{F}}_I)$.
\end{lemma}
\begin{proof}
Fix a winning strategy $\sigma$ for Player II in $\mathcal{G}^\star$. We define a winning strategy $\tilde{\sigma}$ for Player II in $\mathcal{G}(\tilde{\mathcal{F}}_I)$. 

First note that, by Proposition~\ref{prop:covering} and by the elementarity of $M$, for every $q \in \mathrm{Coll}(\omega_1, {<}\kappa) \cap M$ with $q \le p$, there exists $y \in I$ such that for every finite $F \subseteq I \cap ({\uparrow}y)$, $q \not\Vdash \dot{C} \cap F \neq \emptyset$. In particular, since $U$ is cofinal in $I$, this means that for every $q \in \mathrm{Coll}(\omega_1, {<}\kappa) \cap M$ with $q \le p$, there exists an $m \in \omega$ such that for every finite nonempty set $F \subseteq \omega$ with $m \le \min F$, $q \not\Vdash \{u_n : n \in F\} \cap \dot{C} \neq \emptyset$. For every such condition $q$, we let $m(q)$ be the least $m$ that satisfies this property. 

We now construct the strategy $\tilde{\sigma}$ by defining $\tilde{\sigma}(s)$, where $s$ is a finite sequence of natural numbers, inductively on the length of $s$. Simultaneously, we define an auxiliary map $T$ that maps finite sequences of natural numbers to finite decreasing sequences (of the same length) of conditions in $\mathrm{Coll}(\omega_1, {<}\kappa) \cap M$ below $p$. The role of the map $T$ is essentially to translate partial plays of Player I in $\mathcal{G}(\tilde{\mathcal{F}}_I)$ to partial plays of Player I in $\mathcal{G}^\star$.

First let $T(\langle n_0\rangle) = \langle p\rangle$ and  $\tilde{\sigma}(\langle n_0\rangle) = \max(n_0, m(\sigma(\langle p\rangle)))+1$ for every $n_0 \in\omega$. Now suppose that $\tilde{\sigma} \upharpoonright \omega^k$ and $T \upharpoonright \omega^k$ have been defined for some $k > 0$, and fix a sequence $\langle n_0, \ldots, n_k\rangle$ of natural numbers. Denote by $q_{k-1}$ the condition $\sigma(T(\langle n_0, \ldots, n_{k-1}\rangle))$. Note that $q_{k-1} \le p$, as every element of $T(\langle n_0, \ldots, n_{k-1}\rangle)$ is below $p$ by the induction hypothesis and $\sigma$ is a strategy for Player II in $\mathcal{G}^\star$. By definition of $m(q_{k-1})$ and by elementarity of $M$, we can fix a condition $p_k \in \mathrm{Coll}(\omega_1, {<}\kappa) \cap M$ such that $p_k \le q_{k-1}$ and 
\[
p_k \Vdash \{u_n : m(q_{k-1}) \le n < n_k\} \cap \dot{C} = \emptyset.
\]
Let 
\begin{align*}
T(\langle n_0, \ldots, n_k\rangle) &= T(\langle n_0, \dots, n_{k-1}\rangle)^\smallfrown \langle p_k\rangle,\\\tilde{\sigma}(\langle n_0, \ldots, n_k\rangle) &= \max \big(n_k, m(\sigma(T(\langle n_0, \ldots, n_k\rangle)))\big) + 1.
\end{align*}
This completes the definition of the strategy $\tilde{\sigma}$ and the map $T$.

Let us first note that $\tilde{\sigma}$ is a strategy for Player II in $\mathcal{G}(\tilde{\mathcal{F}}_I)$. Indeed, by construction, $\tilde{\sigma}(\langle n_0, \ldots, n_k\rangle) > n_k$ for every finite sequence $\langle n_0, \ldots n_k\rangle$ of natural numbers. 

Now let us show that $\tilde{\sigma}$ is a winning strategy. Fix an infinite sequence $\langle n_k : k\in\omega \rangle$ of natural numbers such that $\tilde{\sigma}(\langle n_0, \ldots, n_{k}\rangle) \le n_{k+1}$ for every $k$. Let $m_k \coloneqq \tilde{\sigma}(\langle n_0, \dots, n_k\rangle)$ and let $p_k$ be the last element of the (finite) sequence $T(\langle n_0, \ldots, n_k\rangle)$ for every $k\in\omega$. Clearly, $T(\langle n_0, \ldots, n_k\rangle) = \langle p_0, \ldots, p_k\rangle$ and $p_0 = p$. Also note that $n_k < m_k \le n_{k+1}$ for every $k\in\omega$. In particular, the sequence $\langle n_k: k\in\omega\rangle$ is strictly increasing. We claim that
\begin{equation}\label{eq:1}
(\omega \setminus n_0) \cap \bigcap_{k\in\omega}\{n \in \omega: p_k \not\Vdash u_n \not\in \dot{C}\} \subseteq \bigcup_{k\in\omega} [n_k, m_k).
\end{equation}
Indeed, pick an $n \ge n_0$ such that $n \not\in \bigcup_{k\in\omega} [n_k, m_k)$. We want to show that $n$ does not belong to the set on the left-hand side of \eqref{eq:1}. Since the sequence $\langle n_k : k \in\omega \rangle$ is strictly increasing, there must be $k\in\omega$ such that $n \in [m_k, n_{k+1})$. By construction,
\[
p_{k+1} \Vdash \{u_i : m(\sigma(\langle p_0, \ldots, p_k\rangle)) \le i < n_{k+1}\} \cap \dot{C} = \emptyset.
\]
Again by construction, $m(\sigma(\langle p_0, \ldots, p_k\rangle)) \le m_k$. In particular, we conclude that $p_{k+1} \Vdash u_n \not\in \dot{C}$, and therefore $n$ does not belong to the left-hand side set of \eqref{eq:1}. Thus, \eqref{eq:1} holds.

By \eqref{eq:1} it suffices to prove
\[
\bigcap_{k\in\omega}\{n \in \omega: p_k \not\Vdash u_n \not\in \dot{C}\} \in \tilde{\mathcal{F}}^+_I,
\] 
as $\bigcup_k [n_k, m_k) \in \tilde{\mathcal{F}}^+_I$ would directly follow.

Note that, by the way we defined $T$, we have $\sigma(\langle p_0, \ldots, p_k\rangle) \ge p_{k+1}$ for every $k$. Since, by assumption, $\sigma$ is a winning strategy for Player II in $\mathcal{G}^\star$, we have
\[
\bigcap_{k\in\omega}\{x \in I : p_k \not\Vdash x \not\in \dot{C}\} \in \mathcal{F}_I^+.
\] 
As $U \in \mathcal{F}_I$, we conclude that $\bigcap_{k\in\omega}\{x \in U : p_k \not\Vdash x \not\in \dot{C}\}$ also belongs to $\mathcal{F}^+_I$, or, equivalently, $\bigcap_{k\in\omega}\{n \in \omega : p_k \not\Vdash u_n \not\in \dot{C}\} \in \tilde{\mathcal{F}}_I^+$, as we wanted to show. Therefore, $\tilde{\sigma}$ is a winning strategy for Player II in $\mathcal{G}(\tilde{\mathcal{F}}_I)$.
\end{proof}

Since $\tilde{\mathcal{F}}_I$ is non-meager, Player II does not have a winning strategy in $\mathcal{G}(\tilde{\mathcal{F}}_I)$ by Laflamme's Theorem~\ref{thm:laflamme}. Hence, by Lemma~\ref{lemma:game}, we know that Player II does not have a winning strategy in $\mathcal{G}^\star$. 

Fix an enumeration $\langle x_k : k\in\omega\rangle$ of $I$. We define a strategy $\sigma$ for Player II in $\mathcal{G}^\star$ as follows: given a finite sequence $\langle p_0, \ldots, p_k\rangle$ of conditions in $\mathrm{Coll}(\omega_1, {<}\kappa) \cap M$ below $p$, we let $\sigma (\langle p_0, \ldots, p_k\rangle)$ be a condition $q \le p_k$ in $M$ such that, for some $y \in I$ with $x_k \preceq y$, $q \Vdash y \in \dot{C}$---note that such $q$ exists by elementarity of $M$, since $p_k \le p \Vdash ``\dot{C}$ is cofinal in $R$".  

Since $\sigma$ is not a winning strategy in $\mathcal{G}^\star$, there is a decreasing sequence  $\langle p_k : k\in\omega\rangle$ of conditions in $\mathrm{Coll}(\omega_1, {<}\kappa) \cap M$ with $p_0 \le p$ and such that $\sigma(\langle p_0, \ldots, p_k\rangle) \ge  p_{k+1}$ for every $k$ and
\[
\bigcup_{k\in\omega} \{x \in I : p_k \Vdash x \not\in \dot{C}\}\in \mathcal{F}_I.
\]

Let $p_\omega = \bigcup_{k\in\omega} p_k$. Since $p_\omega$ extends every $p_k$,  $\{x \in I : p_\omega \Vdash x \not\in \dot{C}\}\in \mathcal{F}_I$. Pick $y \in R$ such that for every $w \in \pi_I[{\uparrow} y]$, $p_\omega \Vdash w \not\in \dot{C}$. Since $p$ forces $\dot{C}$ to be cofinal in $R$, we can also pick $r \le p_\omega$ and $z \in R$ with $y \preceq z$ such that $r \Vdash z \in \dot{C}$. 

Note that, by the way we defined the strategy $\sigma$, $p_\omega$ forces $\dot{C} \cap I$ to be cofinal in $I$. Hence, there exist $r' \le r$ and a $b \in I$ with $\pi_I (z) \preceq b$ such that $r' 
\Vdash b \in \dot{C}$. But $r'$, which extends $p$, also forces $\dot{C}$ to be a meet-subsemilattice of $R$, and thus $r' \Vdash z \wedge b \in \dot{C}$. By Lemma~\ref{lemma:meetproj} and by the way we chose $b$, $z \wedge b = \pi_I(z) \wedge b = \pi_I(z)$. Hence, $r' \Vdash \pi_I (z) \in \dot{C}$. But $r'$ extends $p_\omega$, which forces $\pi_I (z) \not\in \dot{C}$ by the way we chose $y$. This contradiction completes the proof of Theorem~\ref{thm:maind}.

\section{Questions}\label{sec:questions}

We have shown that after collapsing a Mahlo cardinal $\kappa$ via $\mathrm{Coll}(\omega_1, {<}\kappa)$ there are no $3$-ladders of cardinality $\aleph_2$, and consequently, by Corollary~\ref{cor:induction}, no $4$-ladders of cardinality $\aleph_3$.  It is natural to ask what we can say about $4$-ladders of cardinality $\aleph_2$. Are there such lattices in our model? 

Note that the proof of Theorem~\ref{thm:maind} heavily relies on  Theorem~\ref{thm:chainprojection}, and the statement of Theorem~\ref{thm:chainprojection} ceases to hold true if we replace ``breadth $n$" by ``breadth at most $n+1$". For example, it is not hard to modify Ditor's construction of an uncountable $2$-ladder \cite{MR0732199} in order to construct an uncountable $3$-ladder such that for some countably infinite ideal $I \subseteq L$, no member of $\mathcal{F}_I$ is a  chain. Consequently, our proof of Theorem~\ref{thm:maind} cannot be easily adjusted so as to prove the nonexistence of $4$-ladders of cardinality $\aleph_2$ in our model. We conjecture that the reason behind this is that the existence of $4$-ladders of cardinality $\aleph_2$ actually follows from $\mathsf{CH}$, which holds in our model.

\begin{conjecture}
$\mathsf{CH}$ implies the existence of a  $4$-ladder of cardinality $\aleph_2$.
\end{conjecture}

If the conjecture is true, it would suggest that the passage from $3$-ladders to $4$-ladders marks an unexpected qualitative change in behavior at cardinality $\aleph_2$.

Another natural question is whether it is consistent that there is no lower finite lattice of finite breadth and cardinality $\aleph_2$, regardless of the value of the breadth. A related (slightly weaker) question was already raised in \cite[Question 11]{notaro2026maximalladders}.

\begin{question}\label{q:2}
Assuming the consistency of a Mahlo cardinal, is it consistent that there is no lower finite lattice of finite breadth and cardinality $\aleph_2$?
\end{question}

Note that if our conjecture is true, then $\mathsf{CH}$ would necessarily fail in any model of set theory witnessing a positive answer to Question~\ref{q:2}.

Let us also remark that, as far as we know, it is still open whether the existence of a $3$-ladder of cardinality $\aleph_2$ follows from the existence of a lower finite lattice of breadth $3$ and cardinality $\aleph_2$. Note that not every lower finite lattice of breadth $n$ is an $n$-ladder: consider, for example, the diamond lattice $\mathsf{M}_3$, which has breadth $2$ but is not a $2$-ladder.

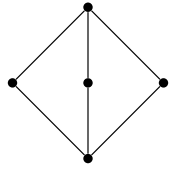
\begin{figure}[H]
\begin{tikzpicture}[
    x=1cm,
    y=1cm,
    vertex/.style={circle,fill,inner sep=1.25pt},
    every path/.style={line width=0.45pt,line cap=round}
]

    \node[vertex] (a) at (0,0) {};
    \node[vertex] (b) at ( -1,1) {};
    \node[vertex] (c) at (0,1) {};
    \node[vertex] (d) at (1,1) {};
    \node[vertex] (e) at (0,2) {};

    \draw
        (a)--(b)
        (a)--(c)
        (a)--(d)
        (b)--(e)
        (c)--(e)
        (d)--(e);
\end{tikzpicture}\vspace{0.5em}
\caption{Hasse diagram of $\mathsf{M}_3$}
\end{figure}

\begin{question}\label{q:3}
Does the existence of a $3$-ladder of cardinality $\aleph_2$ follow from the existence of a lower finite lattice of breadth $3$ and cardinality $\aleph_2$?
\end{question}
Note that a positive answer to Question~\ref{q:3} would imply that the theories (2) and (3) in the statement of Corollary~\ref{cor:main} are equivalent, not just equiconsistent. Furthermore, it follows from Lemma~\ref{lemma:basicproj} that every lower finite lattice of breadth $2$ has a cofinal join-subsemilattice which is a $2$-ladder in the induced ordering. This leads to the following strengthening of Question~\ref{q:3}.
\begin{question}\label{q:4}
Given a lower finite lattice of breadth $3$ and cardinality $\aleph_2$, does it have a cofinal join-subsemilattice which is a $3$-ladder in the induced ordering?
\end{question}

Now let us briefly discuss how Theorem~\ref{thm:main} and its proof are situated in the context of \cite{notaro2026maximalladders}. Our main result answers both Question 1 and Question 5 from \cite{notaro2026maximalladders} in the negative, while Question 2 remains open. The relationship between Theorem~\ref{thm:main} and Question 3 from the same work is more subtle, and in order to discuss it, we need to recall the definition of maximal $n$-ladders: an $n$-ladder is \emph{maximal} if it is not isomorphic to a proper ideal of an $n$-ladder \cite[Definition 2.6]{notaro2026maximalladders}. By the second inequality in Ditor's Theorem~\ref{thm:ditor}, every $n$-ladder of cardinality $\aleph_{n-1}$ is maximal. Moreover, given a forcing $\mathbb{P}$, we say that a maximal $n$-ladder $L$ is \emph{$\mathbb{P}$-indestructible} if $\Vdash_\mathbb{P} ``L$ is maximal"---that is, $\mathbb{P}$ preserves the maximality of the $n$-ladder $L$. Question 3 from \cite{notaro2026maximalladders} asks whether every maximal $3$-ladder of cardinality $\aleph_2$ is indestructible by any $\sigma$-closed forcing. A slight generalization of the concluding argument of the proof of Theorem~\ref{thm:maind}---from the end of the proof of Lemma~\ref{lemma:abs} onward---leads to the following result:
\begin{proposition}
If $P$ is a (maximal) $3$-ladder of cardinality $\aleph_2$ such that $\mathcal{F}_I$ is non-meager for some countably infinite ideal $I \subseteq P$, then $P$ is indestructible by any $\sigma$-closed forcing.
\end{proposition} 
Since not every $3$-ladder of cardinality $\aleph_2$ has a countably infinite ideal $I$ with $\mathcal{F}_I$ non-meager (e.g., a \emph{special} $3$-ladder of cardinality $\aleph_2$ \cite{MR4993406} is a counterexample), the question whether every maximal $3$-ladder is indestructible by any $\sigma$-closed forcing is still open. The preceding proposition nevertheless imposes some constraints on possible witnesses to a negative answer.
\begin{question}[{\cite[Question 3]{notaro2026maximalladders}}]
Is every maximal $3$-ladder indestructible by $\sigma$-closed forcing?
\end{question}

Finally, let us discuss Ditor's Question~A. As noted in Section~\ref{sec:prel:ditor}, if we restrict the question to $\kappa = \aleph_0$, then it has a positive answer when $n = 1,2$ and is independent when $n \ge 3$ (assuming the consistency of a Mahlo cardinal), with the latter statement following from our Theorem~\ref{thm:main}. Moreover, Wehrung \cite{MR2609217} proved that Ditor's Question~A has a positive answer for every $n > 0$ when $\kappa$ is an uncountable regular cardinal. What can be said when $\kappa$ is a singular cardinal? This case is largely open. Ditor singled out the simplest problematic case:

\begin{Dproblem2}
Is there a join-semilattice of breadth $2$ and cardinality $\aleph_{\omega+1}$ whose elements have ${<}\aleph_\omega$ many predecessors?
\end{Dproblem2}

We also expect this second problem to be independent, although establishing this would require large cardinal assumptions far stronger than the existence of a Mahlo cardinal. Indeed, it is shown in \cite{MR4993406} that a positive answer follows from $\square_{\aleph_\omega}$. The failure of $\square_{\aleph_\omega}$ has high consistency strength. A lower bound is given by the consistency of a Woodin cardinal\footnote{If, moreover, $\aleph_\omega$ is a strong limit cardinal, then the failure of $\square_{\aleph_\omega}$ implies $\mathsf{AD}^{\mathbf L(\mathbb R)}$ \cite{MR2194247}, as well as stronger consequences \cite{MR3225589}.} \cite{MR1359965}, while a remarkable recent upper bound is (strictly weaker than) the consistency of a Woodin cardinal that is a limit of Woodin cardinals \cite{blue2026failuresquareuncountablecardinals}.

\printbibliography

\end{document}